\documentclass[reqno]{article}

\usepackage{amsmath,amsthm,amssymb}
\usepackage{bbm}
\usepackage{geometry}
\usepackage{setspace}

\newtheorem{theorem}{Theorem}
\newtheorem{lemma}[theorem]{Lemma}
\newtheorem{corollary}[theorem]{Corollary}
\newtheorem{remark}[theorem]{Remark}

\renewcommand{\epsilon}{\varepsilon}
\newcommand{\eps}{\varepsilon}

\newcommand{\EE}{{\mathbb E}}
\newcommand{\PP}{{\mathbb P}}
\newcommand\Var{\operatorname{Var}}
\newcommand{\RR}{{\mathbb R}}

\newcommand{\cB}{{\mathcal B}}
\newcommand{\cD}{{\mathcal D}}
\newcommand{\cE}{{\mathcal E}}
\newcommand{\cF}{{\mathcal F}}
\newcommand{\cG}{{\mathcal G}}
\newcommand{\cH}{{\mathcal H}}
\newcommand{\cI}{{\mathcal I}}
\newcommand{\cQ}{{\mathcal Q}}
\newcommand{\cZ}{{\mathcal Z}}

\newcommand\nopf{\qed}
\newcommand{\indic}[1]{\mathbbm{1}_{\{#1\}}}
\newcommand\floor[1]{\lfloor {#1} \rfloor}

\newcommand{\DM}{{\Delta M}}
\newcommand{\DY}{{\Delta Y}}

\begin{document}

\title{On the method of typical bounded differences} 
\author{Lutz Warnke%
\thanks{Peterhouse, Cambridge CB2 1RD, UK. 
E-mail: {\tt L.Warnke@dpmms.cam.ac.uk}. 
Part of the work was done while interning with the Theory Group at 
Microsoft Research, Redmond.}}
\date{December 21, 2012}
\maketitle

\begin{abstract}
Concentration inequalities are fundamental tools in probabilistic 
combinatorics and theoretical computer science for proving that 
random functions are near their means. Of particular importance is 
the case where $f(X)$ is a function of independent random variables 
$X=(X_1, \ldots, X_n)$. Here the well known \emph{bounded differences 
inequality} (also called McDiarmid's or Hoeffding--Azuma inequality) 
establishes sharp concentration if the function $f$ does not depend 
too much on any of the variables. 
One attractive feature is that it relies on a very simple Lipschitz 
condition (L): it suffices to show that $|f(X)-f(X')| \leq c_k$ 
whenever $X,X'$ differ only in $X_k$. While this is easy to check, 
the main disadvantage is that it considers \emph{worst-case} changes 
$c_k$, which often makes the resulting bounds too weak to be useful. 

In this paper we prove a variant of the bounded differences inequality 
which can be used to establish concentration of functions $f(X)$ where 
(i) the \emph{typical} changes are small although (ii) the worst case 
changes might be very large. 
One key aspect of this inequality is that it relies on a simple 
condition that (a) is easy to check and (b) coincides with heuristic 
considerations why concentration should hold. Indeed, given an event 
$\Gamma$ that holds with very high probability, we essentially relax 
the Lipschitz condition (L) to situations where $\Gamma$ occurs. The 
point is that the resulting \emph{typical} changes $c_k$ are often 
much smaller than the worst case ones. 

To illustrate its application we consider the reverse $H$-free 
process, where $H$ is $2$-balanced. We prove that the final number of 
edges in this process is concentrated, and also determine its likely 
value up to constant factors. 
This answers a question of Bollob{\'a}s and Erd{\H o}s. 
\end{abstract}

\section{Introduction}
In probabilistic combinatorics and theoretical computer science it is 
often crucial to predict the likely value(s) of a random function. 
More precisely, in many applications $f(X)$ is a function of 
independent random variables $X=(X_1, \ldots, X_N)$, and we need to 
prove that it is concentrated in a narrow range around its expected 
value, i.e., that $f(X)$ typically is about $\mu=\EE f(X)$. 
The crux is that the functions of interest are often defined in an 
indirect or complicated way, so that basic bounds such as Chebychev's 
inequality are either hard to evaluate or give error bounds that are 
too weak in applications. 
In this work we thus investigate easy-to-check conditions which 
ensure that the function $f(X)$ is close to its mean $\mu$ with very 
high probability, i.e., that large deviations from $\mu$ are highly 
unlikely.

An important paradigm in this area of research (see e.g.~\cite{Talagrand1996}) 
states that a random function which depends `smoothly' on many 
independent random variables should be sharply concentrated, meaning 
that $|f(X)-\mu|=o(\mu)$ holds with probability very close to one. 
In many applications (e.g.\ the design of randomized algorithms or 
random graph theory) each random variable $X_k$ takes values in a set 
$\Lambda_k$, and in this case a discrete Lipschitz condition for 
$f:\prod_{j \in [N]}\Lambda_j \to \RR$ conveniently ensures that 
$f(X)$ does not depend too much on any of the variables, where 
$[N]=\{1, \ldots, N\}$. 
Perhaps the most famous result in this context is the \emph{bounded 
differences inequality} (also called McDiarmid's or Hoeffding--Azuma 
inequality), which is nowadays widely used in discrete mathematics 
and computer science, see e.g.\ the surveys~\cite{McDiarmid1989,McDiarmid1998}. 
Here we only state its one-sided version since the analogous lower 
tail estimate $\PP(f(X) \le \mu-t)$ follows by considering the 
function $-f(X)$. 
\begin{theorem}[`Bounded differences inequality']\label{thm:McD}%
{\normalfont\cite{McDiarmid1989}} 
Let $X=(X_1, \ldots, X_N)$ be a family of independent random 
variables with $X_k$ taking values in a set $\Lambda_k$. 
Assume that the function $f:\prod_{j \in [N]}\Lambda_j \to \RR$ 
satisfies the following \emph{Lipschitz condition}: 
\begin{itemize}
\item[(L)] There are numbers $(c_k)_{k \in [N]}$ such that whenever 
$x,\tilde{x} \in \prod_{j \in [N]}\Lambda_j$ differ only in the 
$k$-th coordinate we have 
\begin{equation}\label{eq:fL}
|f(x)-f(\tilde{x})| \le c_k . 
\end{equation}
\end{itemize}
Let $\mu=\EE f(X)$. For any $t \ge 0$ we have 
\begin{equation}\label{McD:Pr}
\PP(f(X) \ge \mu+t) \le \exp\left(-\frac{2t^2}{\sum_{k \in [N]}c_k^2}\right) . 
\end{equation}
\end{theorem}
While the simplicity of (L) makes this inequality very intuitive and 
easy to apply, its perhaps main drawback is that it considers 
\emph{worst case} changes. 
In particular, the resulting concentration bounds are rather weak (or 
even trivial) in situations where the worst case $c_k$ are much 
larger than the \emph{typical} changes. 
A standard example is $f(X)$ counting the number of triangles in the 
binomial random graph $G_{n,p}$: since every pair of vertices has up 
to $n-2$ common neighbours the worst case is $c_k=\Theta(n)$, which 
is much larger than we expect from the $\Theta(np^2)$ common 
neighbours we usually have for $p \geq n^{-1/2+\eps}$. 
In fact, here Theorem~\ref{thm:McD} only gives trivial estimates for 
$p =O(n^{-1/3})$, but it seems plausible that concentration should 
hold in such applications where the typical changes are much smaller 
than the worst case ones.

This motivated a line of research~\cite{KimVu2000,SchudySviridenko2012b,SchudySviridenko2012,Vu2000,Vu2002} 
which focused on tail inequalities of the form 
$\PP(|f(X)-\mu| \geq t) \leq e^{-g(f,X,t)}$ in situations where 
(intuitively speaking) the \emph{average} Lipschitz coefficients 
$c_k$ are small. Pioneered by Kim and Vu~\cite{KimVu2000}, such 
results usually require $f(X)$ to have a special structure (a 
polynomial of independent random variables of a certain type), 
reducing their range of applications compared to \eqref{McD:Pr}. 
Furthermore, the assumptions of such techniques are much more 
involved (and harder to check) than the simple Lipschitz 
condition (L).

In contrast, much less research has been devoted to developing 
easy-to-use tools for proving concentration results in such 
situations. The Hoeffding--Azuma inequality~\cite{Hoeffding1963,Azuma1967} 
implies, for example, that \eqref{McD:Pr} essentially remains true if 
we relax \eqref{eq:fL} to worst case conditional expected changes: 
\begin{equation}\label{eq:fLE} 
|\EE(f(X) \mid X_1, \ldots, X_k)-\EE(f(X) \mid X_1, \ldots, X_{k-1})| \leq c_k . 
\end{equation}
While this might be useful in certain textbook examples, it typically 
has two main drawbacks in involved combinatorial applications: 
(a) conditional expectations are usually difficult to calculate and 
(b) it often yields no substantial improvement (for, say, $k \geq N/2$ 
the worst case in \eqref{eq:fLE} over all choices of $X_1, \ldots, X_{k}$ 
is often comparable to~\eqref{eq:fL}). 
There are also some approaches which allow \eqref{eq:fLE} to be violated 
occasionally~\cite{ChalkerGodboleHitczenkoRadcliffRuehr1999,ChungLu2006,Kim1995a,McDiarmid1998,ShamirSpencer1987}, 
but these usually require knowledge about conditional probability 
distributions, making them particularly difficult to apply when 
$f(X)$ is defined in an indirect or complicated way.

\subsection{Typical bounded differences inequality}\label{sec:TBDI}
In this paper we develop a variant of the bounded differences 
inequality which can be used to establish concentration of functions 
$f(X)$ where (i) the \emph{typical} changes are small although 
(ii) the worst case changes might be very large. 
One key aspect of this inequality is that it relies on a simple and 
attractive condition that (a) is easy to check and (b) coincides with 
heuristic considerations why concentration should hold. 
Indeed, given a `good' event $\Gamma$ that holds with very high 
probability, we essentially relax the Lipschitz condition (L) to 
situations where $\Gamma$ occurs. 
More precisely, for the sake of proving concentration the following 
inequality usually allows us to restrict our attention to such 
\emph{typical} changes, which are often much smaller than the worst 
case ones. 
\begin{theorem}[`Typical bounded differences inequality']\label{thm:McDExt}
Let $X=(X_1, \ldots, X_N)$ be a family of independent random 
variables with $X_k$ taking values in a set $\Lambda_k$. 
Let $\Gamma \subseteq \prod_{j \in [N]}\Lambda_j$ be an event and 
assume that the function $f:\prod_{j \in [N]}\Lambda_j \to \RR$ 
satisfies the following \emph{typical Lipschitz condition}: 
\begin{itemize}
\item[(TL)] There are numbers $(c_k)_{k \in [N]}$ and 
$(d_k)_{k \in [N]}$ with $c_k \le d_k$ such that whenever 
$x,\tilde{x} \in \prod_{j \in [N]}\Lambda_j$ differ only in the 
$k$-th coordinate we have 
\begin{equation}\label{eq:fTL}
|f(x)-f(\tilde{x})| \; \le \; \begin{cases}
		c_k & \;\text{if $x \in \Gamma$,}\\ 
		d_k & \;\text{otherwise.} 
\end{cases}
\end{equation}
\end{itemize} 
For any numbers $(\gamma_k)_{k \in [N]}$ with $\gamma_k \in (0,1]$ 
there is an event $\cB=\cB(\Gamma,(\gamma_k)_{k \in [N]})$ 
satisfying 
\begin{equation}\label{McDExt:PrB}
\PP(\cB) \leq \sum_{k \in [N]} \gamma_k^{-1} \cdot \PP(X \notin \Gamma) 
\quad \text{and} \quad \neg\cB \subseteq \Gamma , 
\end{equation}
such that for $\mu=\EE f(X)$, $e_k=\gamma_k (d_k-c_k)$ and any 
$t \ge 0$ we have 
\begin{equation}\label{McDExt:Pr}
\PP(f(X) \ge \mu+t \text{ and } \neg \cB) \le \exp\left(-\frac{t^2}{2\sum_{k \in [N]}(c_k+e_k)^2}\right) . 
\end{equation}
\end{theorem} 
\begin{remark}\label{rem:McDExt}
If each $X_k$ takes only two values (i.e., when $|\Lambda_k|=2$) the 
exponent in \eqref{McDExt:Pr} may be multiplied by factor of $4$, 
analogous to the standard bound \eqref{McD:Pr}. 
\end{remark}
As before, this inequality is only stated for the upper tail since an 
application to $-f(X)$ yields the same estimate for 
$\PP(f(X) \le \mu-t)$. 
One key property of the `bad' event $\cB$ is that it does \emph{not} 
depend on the function $f(X)$, so that \eqref{McDExt:Pr} can be used 
as a tail estimate in union bound arguments. 
We expect that the \emph{typical} changes $c_k$ are usually 
\emph{substantially smaller} than the worst case $d_k$, and the 
`compensation factor' $\gamma_k$ is supposed to milden the effects of 
the $d_k$ in \eqref{McDExt:Pr}. 
Indeed, in the typical application $\gamma_k$ will be very small 
(this choice is possible if $\Gamma$ holds with very high 
probability), so that we can think of $e_k=\gamma_k (d_k-c_k)$ as a 
negligible `error term'. 
With this in mind, perhaps the most important aspect of 
Theorem~\ref{thm:McDExt} is that it may still yield concentration in 
situations where Theorem~\ref{thm:McD} only gives trivial bounds due 
to very large worst case Lipschitz coefficients.

To illustrate the ease of application of Theorem~\ref{thm:McDExt}, 
consider again the example where $f(X)$ counts the number of 
triangles in $G_{n,p}$. Define $\Gamma$ as the event that every pair 
of vertices has at most $\Delta=\max\{2np^2,n^{\eps}\}$ common 
neighbours, which fails with probability at most 
$e^{-\Omega(n^\eps)}$ by standard Chernoff bounds. It is 
straightforward to see that in this case (TL) holds with, say, 
$c_k=\Delta$ and $d_k=n$. Setting $\gamma_k = n^{-1}$ we thus have 
$e_k =o(c_k)$ and $\PP(\cB) \leq e^{-\Omega(n^\eps)}$, which means 
that both terms are negligible for the sake of establishing 
concentration. It follows that for $p \geq n^{-2/3+\eps}$ the typical 
bounded differences inequality (Theorem~\ref{thm:McDExt}) yields 
tight concentration of the number of triangles (with 
$\PP\left(f(X) \not\in (1 \pm n^{-\eps}) \mu\right) \leq e^{-\Omega(n^\eps)}$, 
say), whereas Theorem~\ref{thm:McD} already fails for $p = n^{-1/3}$. 
Note that we picked all parameters in a uniform way, setting $c_k=C$, 
$d_k=D$ and $\gamma_k=\gamma$; this might be convenient in many 
applications (where $\gamma \approx C/D$ should often suffice).

\subsubsection{Improvement for Bernoulli random variables}\label{sec:01TL}
If the underlying probability space is generated by independent 
Bernoulli random variables we establish much stronger estimates. 
For example, in the common situation where the success probabilities 
are all equal to $p$ (as in $G_{n,p}$) the following natural 
extension of Theorem~\ref{thm:McDExt} essentially allows us to 
multiply the denominator of~\eqref{McDExt:Pr} with an extra factor of 
$p$ (on an intuitive level one can perhaps think of this as applying 
Theorem~\ref{thm:McDExt} after conditioning on $\Theta(Np)$ variables 
being `relevant'). 
\begin{theorem}[`Typical bounded differences inequality for $0$-$1$ variables']\label{thm:McDExtV}
Let $X=(X_1, \ldots, X_N)$ be a family of independent random 
variables with $X_k \in \{0,1\}$ and $p_k=\PP(X_k=1)$. 
Let $\Gamma \subseteq \{0,1\}^{N} $ be an event and assume that the 
function $f:\{0,1\}^{N} \to \RR$ satisfies the typical Lipschitz 
condition (TL) with $\Lambda_k=\{0,1\}$. 
For any numbers $(\gamma_k)_{k \in [N]}$ with $\gamma_k \in (0,1]$ 
there is an event $\cB=\cB(\Gamma,(\gamma_k)_{k \in [N]})$ satisfying 
\eqref{McDExt:PrB} such that for $\mu=\EE f(X)$ and any $t \ge 0$ we 
have 
\begin{equation}\label{McDExtV:Pr}
\PP(f(X) \ge \mu+t \text{ and } \neg \cB) \le \exp\left(-\frac{t^2}{2\sum_{k \in [N]}(1-p_k)p_k(c_k+e_k)^2+2 C t/3 }\right) , 
\end{equation}
where $e_k=\gamma_k (d_k-c_k)$ and $C = \max_{k \in [N]} (c_k+e_k)$. 
\end{theorem} 
\begin{remark}\label{rem:MON}
If $f(X)$ and $\Gamma$ are either both monotone increasing or 
decreasing we have 
\begin{equation}\label{McDExtVMon:Pr}
\PP(f(X) \ge \mu+t) \le \frac{\PP(f(X) \ge \mu+t \text{ and } \neg \cB)}{1-\PP(\cB)} . 
\end{equation}
\end{remark}
In typical applications of this inequality we hope to be able to 
ignore the `error term' $2 C t/3$ (and select $\gamma_k$ such that 
$c_k+e_k\approx c_k$, as before). In this case \eqref{McDExtV:Pr} is 
close $e^{-t^2/(2\sum p_k c_k^2)}$, which for $p_k=o(1)$ is a 
significant improvement of the corresponding $e^{-2t^2/(\sum c_k^2)}$ 
from Remark~\ref{rem:McDExt}. 
For example, in the case of triangles in $G_{n,p}$ this allows us to 
extend the concentration result of the previous section to edge 
probabilities satisfying $p \ge n^{-4/5+\eps}$. 
In fact, the estimates implied by \eqref{McDExtVMon:Pr} are sometimes 
comparable to those of Janson's inequality~\cite{Janson1990,RiordanWarnke2012J}, 
see Section~\ref{sec:Janson}.

Ignoring the `good' event $\Gamma$ in Theorem~\ref{thm:McDExtV} we 
also obtain a strengthening of Theorem~\ref{thm:McD}. Since this 
natural variant of the bounded differences inequality does not seem 
to be as widely known, we explicitly state it for ease of reference 
(if each $(1-p_k)p_k$ is weakened to $\max_{i}\min\{1-p_i,p_i\}$ then 
\eqref{McDExtVCor:Pr} follows from Theorem~3.9 in McDiarmid's 
survey~\cite{McDiarmid1998}; Alon, Kim and Spencer~\cite{AlonKimSpencer1997} 
also proved a comparable inequality that applies to small values of~$t$ 
only: for those the contribution of $Ct$ to the denominator 
of~\eqref{McDExtVCor:Pr} is negligible). 
\begin{corollary}[`Bounded differences inequality for $0$-$1$ variables']\label{Cor:McDExtVCor}
Let $X=(X_1, \ldots, X_N)$ be a family of independent random 
variables with $X_k \in \{0,1\}$ and $p_k=\PP(X_k=1)$. 
Assume that the function $f:\{0,1\}^{N} \to \RR$ satisfies the 
Lipschitz condition (L) with $\Lambda_k=\{0,1\}$. 
Let $\mu=\EE f(X)$ and $C=\max_{k \in [N]}c_k$. For any $t \ge 0$ we 
have 
\begin{equation}\label{McDExtVCor:Pr}
\PP(f(X) \ge \mu+t) \le \exp\left(-\frac{t^2}{2\sum_{k \in [N]}(1-p_k)p_kc_k^2+2 C t/3 }\right) . 
\end{equation}
\end{corollary} 
\begin{proof}
Apply Theorem~\ref{thm:McDExtV} with $\Gamma=\{0,1\}^{N}$ and $d_k=c_k$. 
\end{proof}
This extends Bernstein's inequality (a strengthening of the Chernoff 
bounds for small deviations, see e.g.\ Remark~2.9 in \cite{JLR}), 
which applies to sums of independent random variables. One key aspect 
of~\eqref{McDExtVCor:Pr} is that it is almost tight when 
$f(X)=\sum_k X_k$, in which case $V=\Var f(X) = \sum_k p_k(1-p_k)$ 
and $c_k=1$. Indeed, the estimate of Corollary~\ref{Cor:McDExtVCor} 
is then close to $e^{-t^2/(2V)}$ for $t$ not too large, which is 
exactly the tail behaviour predicted by the central limit theorem. 
\begin{remark}\label{rem:McDExtB}
Our arguments in fact yield a slightly stronger form of 
\eqref{McDExtV:Pr}--\eqref{McDExtVCor:Pr}, analogous to Bennet's 
sharpening of the Chernoff bounds (see e.g.\ Remark~2.9 in \cite{JLR}). 
Indeed, for $\phi(x)=(1+x)\log(1+x)-x$ we can improve terms of the 
form $e^{-t^2/(2V+2Ct/3)}$ to $e^{-V/C^2 \cdot \phi(Ct/V)}$, where 
$V$ equals $\sum_k (1-p_k)p_k(c_k+e_k)^2$ and $\sum_k (1-p_k)p_kc_k^2$ 
in \eqref{McDExtV:Pr} and \eqref{McDExtVCor:Pr}. 
For $t=\omega(V/C)$ these refined estimates sharpen the exponents 
from order $\Theta(t/C)$ to $\Theta(t/C \cdot \log(Ct/V))$, i.e., 
yield a logarithmic improvement. 
\end{remark}
\begin{remark}\label{rem:McDExtV}
Theorem~\ref{thm:McDExtV} and Corollary~\ref{Cor:McDExtVCor} extend 
with minor modifications to the case where each $X_k$ takes values in 
a set $\Lambda_k$ and satisfies $\max_{\eta \in \Lambda_k} \PP(X_k=\eta) \ge 1-p_k$. 
Indeed, \eqref{McDExtV:Pr} and \eqref{McDExtVCor:Pr} both hold after 
deleting $(1-p_k)$ and replacing $c_k+e_k$ with $c_k+e_k \cdot (1-p_k)^{-1}$. 
\end{remark}

\subsubsection{Two-sided Lipschitz conditions}\label{sec:TTL}
The typical Lipschitz condition (TL) is `one-sided': 
$|f(x)-f(\tilde{x})| \leq c_k$ is supposed to hold if $x \in \Gamma$. 
This keeps the formulas simple, but in many applications it is easier 
(and perhaps more natural) to verify a `two-sided' condition where 
$x,\tilde{x} \in \Gamma$ holds. The following theorem states that we 
may use a two-sided variant of (TL) at the cost of slightly 
increasing the `error term' $e_k$. 
\begin{theorem}[`Two-sided typical Lipschitz condition']\label{thm:McDfTLP}
Theorems~\ref{thm:McDExt}, \ref{thm:McDExtV} and Remarks~\ref{rem:McDExt}, \ref{rem:McDExtB}, \ref{rem:McDExtV} 
remain valid with $e_k=2\gamma_k (d_k-c_k) q_k^{-1}$ and $\min_{\eta \in \Lambda_k} \PP(X_k=\eta) \geq q_k$ 
if the Lipschitz condition \eqref{eq:fTL} of (TL) is replaced by the 
following two-sided variant: 
\begin{equation}\label{eq:fTLP}
|f(x)-f(\tilde{x})| \; \le \; \begin{cases}
		c_k & \;\text{if $x,\tilde{x} \in \Gamma$,}\\ 
		d_k & \;\text{otherwise.} 
\end{cases}
\end{equation} 
\end{theorem}
Whenever $q_k^{-1}$ is not too big \eqref{eq:fTLP} seems the most 
convenient condition: it is much simpler to check than \eqref{eq:fTL} 
and does not substantially deteriorate the error bounds. 
For example, in the random graph $G_{n,p}$ we usually have 
$q_k^{-1} \leq n^{2}$, which in the typical application with 
$\PP(X \not\in \Gamma) \leq n^{-\omega(1)}$ can be compensated by 
adapting $\gamma_k$ accordingly (also note that 
$(1-p_k)p_k q_k^{-1} \leq 1$ in case of Theorem~\ref{thm:McDExtV}). 
\begin{remark}
As pointed out by Oliver Riordan it is possible to bootstrap 
\eqref{eq:fTLP} from \eqref{eq:fTL} by modifying the good event. 
Indeed, defining $\Gamma' \subseteq \Gamma$ such that for 
$x \in \Gamma'$ any single coordinate change results in a sample 
point satisfying $\tilde{x} \in \Gamma$, it follows that the 
one-sided condition $x \in \Gamma'$ implies the two-sided condition 
$x,\tilde{x} \in \Gamma$. 
Using the bound $\PP(X \notin \Gamma') \leq \sum_{k \in [N]} q_k^{-1} \cdot \PP(X \notin \Gamma)$ 
this approach often leads to estimates that are comparable with 
Theorem~\ref{thm:McDfTLP} (in fact, monotonicity of $\Gamma$ also 
transfers to $\Gamma'$). 
\end{remark}

In some applications (a) the $q_k$ are very small and (b) exploiting 
$x \in \Gamma$ when bounding $|f(x)-f(\tilde{x})|$ is difficult, in 
which case neither the two-sided \eqref{eq:fTLP} nor the one-sided 
\eqref{eq:fTL} seem to be suitable. 
In an attempt to deal with such situations we introduce an 
intermediate variant, which is `locally' two-sided: it only requires 
the (one-sided) typical Lipschitz condition \eqref{eq:fTL} to hold 
when each coordinate of \emph{both} sample points $x,\tilde{x}$ 
satisfies some local `good' event $x_j,\tilde{x}_j \in \Gamma_j$. 
\begin{theorem}[`Typical bounded differences inequality with truncation']\label{thm:McDExtP}
Let $X=(X_1, \ldots, X_N)$ be a family of independent random 
variables with $X_k$ taking values in a set $\Lambda_k$. 
Suppose $(\Gamma_k)_{k \in [N]}$ and $\Gamma \subseteq \prod_{j \in [N]}\Gamma_j$ 
are events with $\Gamma_k \subseteq \Lambda_k$. 
Assume that the function $f:\prod_{j \in [N]}\Lambda_j \to \RR$ 
satisfies the Lipschitz condition \eqref{eq:fTL} of (TL) only for all 
$x,\tilde{x} \in \prod_{j \in [N]}\Gamma_j$ that differ only in the 
$k$-th coordinate, and that $|f(x)-f(\tilde{x})| \le s$ for all 
$x,\tilde{x} \in \prod_{j \in [N]}\Lambda_j$. 
For any numbers $(\gamma_k)_{k \in [N]}$ with $\gamma_k \in (0,1]$ 
there is an event $\cB=\cB(\Gamma,(\gamma_k)_{k \in [N]})$ satisfying 
\eqref{McDExt:PrB} such that for $\mu=\EE f(X)$, $\Delta =s\PP(X \notin \Gamma)$, 
$e_k=\gamma_k(d_k-c_k)$ and any $t \ge 0$ we have 
\begin{equation}\label{McDExtP:Pr}
\PP(f(X) \ge \mu+t+\Delta \text{ and } \neg \cB) \le \exp\left(-\frac{t^2}{2\sum_{k \in [N]}(c_k+e_k)^2}\right) . 
\end{equation}
\end{theorem} 
\begin{remark}\label{rem:McDExtP}
If $\Gamma = \prod_{j \in [N]}\Gamma_j$ holds we may set 
$\cB = \neg \Gamma$, $e_k=0$ and multiply the exponent of 
\eqref{McDExtP:Pr} by a factor of $4$. 
If certain monotonicity properties hold we can remove the $\Delta$ 
term: for example, we can set $\Delta=0$ if $f(x) \geq f(\tilde{x})$ 
whenever $x,\tilde{x} \in \prod_{j \in [N]}\Lambda_j$ differ only in 
their $k$-th coordinates $x_k \in \Lambda_k \setminus \Gamma_k$ and 
$\tilde{x}_k \in \Gamma_k$. 
\end{remark}
Theorem~\ref{thm:McDExtP} seems particularly useful when the 
underlying random variables are grouped into larger blocks $B_k$, so 
that each $X'_k$ now takes values in its own product space 
$\Lambda'_k=\prod_{j \in B_k}\Lambda_{j}$ (by construction the $X'_k$ 
are again independent). 
For example, the so-called `vertex exposure' of $G_{n,p}$ uses $n-1$ 
blocks, where $X'_k$ corresponds to the group of edges 
$E_k=(v_{k}v_{k+1}, \ldots, v_{k}v_{n})$. 
In this case $q_k \leq p^{n-k}$ and the `good' event $\Gamma$ of, 
say, having at most $\Sigma=\max\{2np,n^{\eps}\}$ neighbours can 
dramatically fail after changing the $k$-th coordinate (the degree of 
$v_k$ can change up to $n-k$). 
Here we can overcome these issues using the `local' event $\Gamma_k$ 
that at most $\Sigma$ edges of $E_k$ are present, so that after a 
one-coordinate change of $x \in \prod_{j}\Gamma_j$ from $x_k$ to 
$\tilde{x}_k \in \Gamma_k$ the degree of every vertex changes by at 
most $\Sigma$ (if $x \in \Gamma$ then every vertex has at most 
$2\Sigma$ neighbours). 
In other words, the local $\Gamma_k$ and global $\Gamma$ can 
complement each other in order to milden the large worst case effects 
(in particular when many variables are associated with each 
coordinate).

Theorem~\ref{thm:McDExtP} also allows us to routinely apply certain 
truncation arguments (without ad-hoc calculations). A typical example 
is $f(X)=\sum_{k }X_k$ with $X_k$ having exponential tails, where one 
often first proves concentration of, say, $\sum_{k}\min\{X_k,C \log N\}$, 
and then transfers this result to the original sum, see e.g.~\cite{Achlioptas2000,BorgsChayesMertensPittel2004}. 
Here \eqref{McDExtP:Pr} almost immediately yields concentration of 
$f(X)$ via the local events $\Gamma_k$ that $X_k \leq C \log N$ 
occurs (setting $\Gamma = \prod_{j}\Gamma_j$, $d_k=c_k=C \log N$ and 
$\gamma_k=1$).

\subsubsection{Dynamic exposure of the variables}\label{sec:AQ} 
The previous inequalities can be refined by exposing the values of 
the random variables $X_i$ one by one in an \emph{adaptive} order. 
Intuitively this allows us to exploit that after having learned the 
values of certain variables, some other $X_j$ may not any more 
influence the value of $f(X)$. This approach was introduced by Alon, 
Kim and Spencer~\cite{AlonKimSpencer1997}, and is particularly useful 
whenever we can determine $f(X)$ without knowing the value of all 
random variables. 
More formally, a \emph{strategy} sequentially exposes $X_{q_1}, X_{q_2}, \ldots$, 
where each index $q_{i}=q_i(X_{q_{1}}, \ldots, X_{q_{i-1}})$ may 
depend on the previous outcomes and indices (we use the convention 
that $q_{k+1}=q_{k}$ if $f(X)$ is determined by $(X_{q_1}, \ldots, X_{q_k})$ 
with $k < N$); every strategy has a natural representation in form of 
a decision tree. 
With a fixed strategy in mind, for every possible outcome $X=(X_1, \ldots, X_N)$ 
we obtain a set of queried indices $Q \subseteq [N]$, and by $\cQ$ we 
denote the set of all possible such query sets $Q$. 
The resulting key improvement is that in most inequalities we 
essentially may replace $k \in [N]$ with $k \in Q$ for some `worst 
case' set of indices $Q \in \cQ$ (note that $\gamma_k=\gamma$ is a 
typical choice in applications). 
\begin{theorem}[`Dynamic exposure of the variables']\label{thm:DynImpr}
Suppose that $\gamma_k=\gamma$ for all $k \in [N]$. 
For any strategy Theorems~\ref{thm:McD}, \ref{thm:McDExt}, \ref{thm:McDExtV}, \ref{thm:McDfTLP}, \ref{thm:McDExtP}, 
Corollary~\ref{Cor:McDExtVCor} and Remarks~\ref{rem:McDExt}, \ref{rem:McDExtB}, \ref{rem:McDExtV}, \ref{rem:McDExtP} 
remain valid with $\sum_{k \in [N]}$ replaced by 
$\max_{Q \in \cQ} \sum_{k \in Q}$ and $\max_{k \in [N]}$ replaced by 
$\max_{Q \in \cQ} \max_{k \in Q}$, with the addition that $\cB$ 
depends on the query strategy. 
\end{theorem}
\begin{theorem}[`Monotone dynamic exposure of the variables']\label{thm:DynImpr2}
Consider any strategy satisfying $q_{i+1} \geq q_{i}$ in each step. 
Then Theorems~\ref{thm:McD}, \ref{thm:McDExt}, \ref{thm:McDExtV}, \ref{thm:McDfTLP}, \ref{thm:McDExtP}, 
Corollary~\ref{Cor:McDExtVCor} and Remarks~\ref{rem:McDExt}, \ref{rem:MON}, \ref{rem:McDExtB}, \ref{rem:McDExtV}, \ref{rem:McDExtP} 
remain valid with $\sum_{k \in [N]}$ replaced by 
$\max_{Q \in \cQ} \sum_{k \in Q}$ and $\max_{k \in [N]}$ replaced by 
$\max_{Q \in \cQ} \max_{k \in Q}$, with the exception that 
\eqref{McDExt:PrB} remains unchanged. 
\end{theorem}
Applied to Corollary~\ref{Cor:McDExtVCor}, Theorem~\ref{thm:McDExtV} 
and Remark~\ref{rem:McDExtB} these results tighten and extend an 
inequality of Alon, Kim and Spencer~\cite{AlonKimSpencer1997}, which 
is based on the Lipschitz condition (L). 
In certain applications dynamic exposure yields significant improvements, 
and for an illustrating example we refer to Claim~2 in~\cite{AlonKimSpencer1997}, 
where it is crucial to reduce (the order of magnitude of) the number 
of queried variables. 
Further refinements are possible by using adaptive Lipschitz bounds 
$c_k$, which is perhaps most easily exploited by tailoring the 
arguments of Section~\ref{sec:proof} to the specific application.

One key feature of Theorem~\ref{thm:DynImpr2} is that the `bad' event 
$\cB$ does not depend on the strategy used, making it particularly 
useful in union bound arguments. 
As an illustration consider the example where $f(X)=f_U(X)$ counts, 
in $G_{n,p}$, the number of triangles in a subset $U \subseteq V$ of 
the vertices. Define $\Gamma$ as in Section~\ref{sec:TBDI}. 
Since $f(X)$ depends only on edges in $U$, using Theorems~\ref{thm:McDExtV} 
and~\ref{thm:DynImpr2} (sequentially exposing all edges in $U$) 
we infer for $|U|=u\geq u_0=u_0(n,p)$ and 
$\tilde{\Delta}=\min\{\Delta,u\}$ that 
\[
\PP(f(X) \not\in(1 \pm n^{-\eps}) \mu \text{ and } \neg\cB) \leq \exp\left(-\Theta\left(\frac{n^{-2\eps}u^{6}p^{6}}{u^2p {\tilde{\Delta}}^2 + n^{-\eps}u^3p^3\tilde{\Delta}}\right)\right) \leq n^{-\omega(u)} . 
\]
Taking a union bound the probability that some $U \subseteq V$ with 
$|U| \geq u_0$ has the `wrong' number of triangles is at most 
$\PP(\cB) + n^{-\omega(u_0)}$. Here we crucially exploited that $\cB$ 
is a `global' event not depending on $f(X)$ or $U$, so that 
$\PP(\cB) \le e^{-\Omega(n^{\eps})}$ does not need to `compete' with 
the $n^{u}$ choices for the subsets (this issue often makes 
traditional bad events ineffective in union bound arguments).

\subsubsection{Weakening the independence assumption}\label{sec:GL} 
The concentration results discussed so far extend to certain 
dependent random variables $X=(X_1, \ldots, X_N)$ that are generated 
by a sequence of `nearly' independent (or uniform) random choices. 
As we shall see, they e.g.\ apply to random permutations 
$\pi \in S_n$ and uniform random graphs $G_{n,m}$. 
To motivate the new (GL) condition below we consider independent 
random variables, in which case the mapping $\rho_k:\Sigma_a \to \Sigma_b$ 
that changes the value of the $k$-th coordinate from $a$ to $b$ is 
a bijection. Here (L) yields $\left|f(x)-f(\rho_k(x))\right| \le c_k$ 
and independence implies $\PP(X=x \mid X \in \Sigma_a)= \PP(X=\rho_k(x) \mid X \in \Sigma_b)$. 
With this in mind \eqref{eq:fGTL} can be viewed as a natural analogue 
of \eqref{eq:fTL} in which the outcomes $x$ and $\tilde{x}=\rho_k(x)$ 
may differ in more than just one coordinate, and \eqref{eq:PGTL} 
accounts for the fact that the variables are not necessarily 
independent. 
\begin{theorem}[`General bounded differences inequality']\label{thm:McDGenExt}
Let $X=(X_1, \ldots, X_N)$ be a family of random variables with $X_k$ 
taking values in a set $\Lambda_k$. 
Let $\Gamma \subseteq \prod_{j \in [N]}\Lambda_j$ be an event. 
Then the conclusions of Theorem~\ref{thm:McDExt} remain valid if 
instead of (TL) the function $f:\prod_{j \in [N]}\Lambda_j \to \RR$ 
satisfies the following \emph{general Lipschitz condition}: 
\begin{itemize}
\item[(GL)] 
There are numbers $(c_k)_{k \in [N]}$ and $(d_k)_{k \in [N]}$ with 
$c_k \le d_k$ such that the following holds for any two possible 
sequences of outcomes $a_1, \ldots, ,a_{k-1},a$ and 
$a_1, \ldots, a_{k-1},b$ of $X_1, \ldots, X_k$. Defining 
\begin{equation*}\label{eq:Sigma}
\Sigma_z = \Bigl\{x=(a_1, \ldots, a_{k-1},z,x_{k+1}, \ldots, x_N) \in \prod_{j \in [N]}\Lambda_j ~:~ \text{$\PP(X=x)>0$}\Bigr\} , 
\end{equation*}
there is an injection $\rho_k=\rho_k(\Sigma_a,\Sigma_b):\Sigma_a \to \Sigma_b$ 
such that for all $x \in \Sigma_a$ we have 
\begin{gather}
\label{eq:fGTL}
\left|f(x)-f(\rho_k(x))\right| \; \le \; \begin{cases}
		c_k & \;\text{if $x \in \Gamma$,}\\ 
		d_k & \;\text{otherwise.} 
\end{cases}
 \quad {and }\\ 
\label{eq:PGTL}
 \PP(X=x \mid X \in \Sigma_a) \leq \PP(X=\rho_k(x) \mid X \in \Sigma_b) . 
\end{gather}
\end{itemize} 
\end{theorem} 
\begin{remark}\label{rem:GTL}
The proof shows that $\rho_k$ must be a bijection with equality 
in \eqref{eq:PGTL}. 
Furthermore, if $X_k$ takes at most two values conditioned on 
$X_1, \ldots, X_{k-1}$, then the exponent in \eqref{McDExt:Pr} may be 
multiplied by factor of $4$.
In fact, \eqref{McD:Pr} holds if $\Gamma =\prod_{j \in [N]}\Lambda_j$ 
(or $r_k=0$ below). 
In addition, for \eqref{McDExt:Pr} to hold with $e_k=\gamma_k r_k \ge 0$ 
it suffices if we relax (GL) to the average Lipschitz condition 
\begin{equation}\label{eq:EfGTL}
|\EE(f(X)|X \in \Sigma_a) - \EE(f(X) \mid X \in \Sigma_b)| \le c_k+r_k \PP(X \not\in \Gamma \mid X \in \Sigma_a) .
\end{equation}
\end{remark}
To illustrate the application of the (GL) condition we consider 
uniform permutations $\pi \in S_n$, which are generated by 
sequentially choosing each $\pi(k)$ randomly from 
$[n] \setminus \{\pi(1), \ldots, \pi(k-1)\}$. Here $\Sigma_z$ 
contains all $\pi$ with $\pi(k)=z$ and $\pi(j)=a_j$ for 
$1 \le j < k$. In this case a bijection $\rho_k:\Sigma_a \to \Sigma_b$ 
is defined by the transposition of $a$ and $b$, so that 
$\pi'=\rho_k(\pi)$ satisfies $\pi'(k)=b$, $\pi'(\pi^{-1}(b))=a$ and 
$\pi'(i)=\pi(i)$ for $\pi(i) \notin \{a,b\}$. Using 
$|\Sigma_a|=|\Sigma_b|=(n-k)!$ and the uniform measure it is not hard 
to check that \eqref{eq:PGTL} holds with equality. 
We see that for establishing \eqref{eq:fGTL} it suffices to bound 
$|f(\pi)-f(\pi')|$ whenever $\pi$ and $\pi'$ are related via a 
transposition, which is an intuitive and easy to check condition 
(this may correspond to changing two coordinates).

One key aspect of (GL) is that it often maintains the simplicity of 
(L) and (TL). Here uniform probability measures are particularly 
convenient, for which it suffices to first define bijections 
$\rho_k:\Sigma_a \to \Sigma_b$ and then check \eqref{eq:fGTL} only 
(using $\PP(X=x \mid X \in \Sigma_a)=\PP(X=x)/\PP(X \in \Sigma_a)$ 
these must satisfy \eqref{eq:PGTL} with equality). 
Indeed, extending the permutations example, for random sequences 
$T=(t_1, \ldots, t_m)$ of $m$ distinct elements from $W$ it is enough 
to estimate $|f(T)-f(T')|$ whenever both sequences are related by 
changing one coordinate (i.e., $t_k \neq t'_k$) or interchanging the 
order of two coordinates (i.e., $t_k =t'_j$ and $t_j=t'_k$). Note 
that this example includes the random graph process and various 
hypergraph processes as special cases. 
Since every set with $m$ elements gives rise to $m!$ ordered 
sequences, the above result also readily carries over to uniform 
random subsets $S \subseteq W$ of size $|S|=m$: it suffices to bound 
$|f(S)-f(S')|$ whenever the sets are minimally different, i.e., 
satisfy $|S \cap S'|=m-1$ (note that for $m > |W|/2$ better results 
are obtained by choosing the complement uniformly at random). 
Here the uniform random graph $G_{n,m}$ and uniform hypergraphs are 
special cases. 
Note that the above construction also extends to multiple 
(independent) random objects; for example, if $M$ random subsets 
$X=(S_1, \ldots, S_M)$ with $S_i \in W_i$ and $|S_i|=m_i$ are chosen 
independently it suffices to consider $|f(X)-f(X')|$ only for the 
cases where $X$ and $X'$ are minimally different in one set. 
Finally, with similar easy-to-check conditions (GL) also applies, for 
example, to finite metric spaces, perfect matchings and the 
configuration model $G^*_{\mathrm{d}}$, see 
e.g.~\cite{MilmanSchechtman1986,Wormald1999Regular,BollobasRiordan2012}.

Several extensions of Theorem~\ref{thm:McDExt} carry over to 
Theorem~\ref{thm:McDGenExt} with some minor modifications, and 
results analogous to those of Sections~\ref{sec:01TL} and~\ref{sec:TTL}, 
including a two-sided Lipschitz condition, are stated below 
(Remark~\ref{rem:McDExtB} also applies to \eqref{McDGenV:Pr} after 
adjusting $V$ accordingly). 
\begin{theorem}[`General bounded differences inequality for asymmetric variables']\label{thm:McDGenV}
Let $X=(X_1, \ldots, X_N)$ be a family of random variables with $X_k$ 
taking values in a set $\Lambda_k$, where 
$\max_{\eta \in \Lambda_k} \PP(X_k=\eta \mid X_1, \ldots, X_{k-1}) \ge 1-p_k$ 
holds. 
Let $\Gamma \subseteq \prod_{j \in [N]}\Lambda_j$ be an event and 
assume that the function $f:\prod_{j \in [N]}\Lambda_j \to \RR$ 
satisfies the general Lipschitz condition (GL). 
For any numbers $(\gamma_k)_{k \in [N]}$ with $\gamma_k \in (0,1]$ 
there is an event $\cB=\cB(\Gamma,(\gamma_k)_{k \in [N]})$ satisfying 
\eqref{McDExt:PrB} such that for $\mu=\EE f(X)$ and any $t \ge 0$ we 
have 
\begin{equation}\label{McDGenV:Pr}
\PP(f(X) \ge \mu+t \text{ and } \neg \cB) \le \exp\left(-\frac{t^2}{2\sum_{k \in [N]}p_k\left(c_k+e_k \cdot (1-p_k)^{-1} \right)^2+2 C t/3}\right) ,
\end{equation}
where $e_k=\gamma_k (d_k-c_k)$ and $C = \max_{k \in [N]} (c_k+e_k)$. 
\end{theorem} 
\begin{theorem}[`Two-sided general Lipschitz condition']\label{thm:McDGenfTL}
Theorems~\ref{thm:McDGenExt}, \ref{thm:McDGenV} and 
Remark~\ref{rem:GTL} remain valid with $e_k=2\gamma_k (d_k-c_k) q_k^{-1}$ 
and $\min_{\eta \in \Lambda_k} \PP(X_k=\eta \mid X_{1}, \ldots, X_{k-1}) \geq q_k$ 
if the Lipschitz condition \eqref{eq:fGTL} of (GL) is replaced by the 
following two-sided variant: 
\begin{equation}\label{eq:fGTLP}
|f(x)-f(\rho_k(\tilde{x}))| \; \le \; \begin{cases}
		c_k & \;\text{if $x,\rho_k(\tilde{x}) \in \Gamma$,}\\
		d_k & \;\text{otherwise.}
\end{cases}
\end{equation} 
In addition, $q_k \leq |\Lambda_k|^{-1}$ suffices when all possible 
outcomes occur with the same probability. 
\end{theorem}
The sufficient condition $q_k \leq |\Lambda_k|^{-1}$ often makes the 
two-sided Lipschitz condition of Theorem~\ref{thm:McDGenfTL} easy to 
apply. For example, in case of random permutations $\pi \in S_n$ and 
random graphs $G_{n,m}$ (or the random graph process) we may take 
$q_k=n^{-1}$ and $q_k=n^{-2}$, respectively.

\subsection{Discussion and applications} 

\subsubsection{A wider perspective} 
As discussed, in probabilistic combinatorics and the analysis of 
randomized algorithms we frequently need to prove that a random 
function is not too far from its mean, e.g., that $f(X) \approx \mu$ 
or $f(X) \leq 2 \mu$ holds. 
A common feature of many recent applications is that the functions of 
interest are only `smooth enough' on a \emph{high probability} event, 
whereas their \emph{deterministic} worst case changes are too large 
for the standard bounded differences inequality 
(Theorem~\ref{thm:McD}) to be effective.

In these cases there is no general method, but in the past certain 
ad-hoc arguments have been successfully used in such situations (see 
e.g.~\cite{Bollobas1988Chrom,BollobasBrightwell1992,KimVu2004,KrivelevichLubetzkySudakov2012,RiordanWarnke2012AP}). 
Usually the key idea is to construct a random function $g(X)$ that is 
a \emph{smooth approximation} of $f(X)$, which in particular by 
definition ensures that the Lipschitz coefficients are always small 
(approximation usually means that $f(X) \approx g(X)$ holds with high 
probability, but often $\EE f(X) \approx \EE g(X)$ is also needed). 
Here smoothness makes it possible to apply concentration inequalities 
(often the bounded differences inequality) to $g(X)$, whereas the 
approximation property ensures that concentration transfers from 
$g(X)$ to $f(X)$. The main disadvantage of this approach is that it 
relies on ad-hoc arguments (which can be involved); in particular, 
finding suitable approximation functions may require ingenuity.

One aim of this paper is to provide easy-to-apply tools which can 
\emph{routinely} deal with such situations, establishing 
concentration in a rather simple way. 
For example, in the frequent case where the good event $\Gamma$ holds 
with probability at least $1-N^{-\omega(1)}$ we can typically choose 
$\gamma_k^{-1}=\max |f(X)|$ and then completely ignore the worst case 
effects, see e.g.\ the proof of Theorem~\ref{thm:Hcon} 
(this approach also applies, for example, to Lemma 14 in~\cite{RiordanWarnke2012AP} 
and parts of the martingale-based proof of Theorem 2.2 in~\cite{KrivelevichLubetzkySudakov2012}). 
In other words, the crucial advantage of our new inequalities is that 
they can often remove the need for sometimes difficult ad-hoc 
arguments using only a minimum amount of calculations (which 
typically even coincide with heuristic considerations).

\subsubsection{Comparison with Janson's inequality}\label{sec:Janson}
In this section we demonstrate that in certain applications our 
inequalities give exponential estimates that (i) are tight and 
(ii) successfully compete with the well known Janson's inequality. 
To this end we focus on subgraph counts in the binomial random graph 
$G_{n,p}$ since a concrete example seems more illustrative to us. 
Henceforth we assume that $H$ is a fixed \emph{$2$-balanced} graph, 
i.e., where $H$ has $e_H \geq 2$ edges and all its proper subgraphs 
$G \subsetneq H$ with $v_G \geq 3$ vertices satisfy 
\begin{equation}\label{eq:d2}
\frac{e_G-1}{v_G-2} \leq \frac{e_H-1}{v_H-2} = d_2(H) . 
\end{equation}
This class of graphs includes, for example, complete graphs and 
cycles of arbitrary size. Let $Y_H$ count the number of $H$ copies 
in $G_{n,p}$. For $2$-balanced graphs it is well-known (see 
e.g.~\cite{JLR}) that Janson's inequality gives
\begin{equation}\label{eq:JansonEx}
\PP(Y_H \leq \mu -t ) \leq \exp\left(-\Theta\left(\frac{t^2}{\mu^2/(n^2p)}\right)\right) 
\end{equation}
for $p \geq n^{-1/d_2(H)}$. 
Spencer~\cite{Spencer1990} proved that, assuming 
$p \geq n^{-1/d_2(H)} (\log n)^b$ with $0 < b=b(H) < 2$, for every 
$c>0$ the following holds with probability at least $1-n^{-c}$ for 
$n \geq n_0(c,H)$: every pair $xy$ of vertices is contained in at 
most $\Delta=O(n^{v_H-2}p^{e_H-1})$ extensions to copies of $H$ (for 
which adding the edge $xy$ completes a copy of $H$ containing $xy$). 
The latter event will be our decreasing~$\Gamma$, which allows us to 
use $c_k=\Theta(n^{v_H-2}p^{e_H-1})=\Theta(\mu/(n^2p))$ as well as 
$d_k=n^{v_H}$ and $\gamma_k=n^{-v_H}$ in our typical bounded 
differences inequality, so that $e_k=\gamma_k (d_k-c_k) =o(c_k)$. 
Applying Spencer's result with $c=v_H+3$ we have 
$\PP(\cB) \leq n^{-1}$ by~\eqref{McDExt:PrB}. Note that, since for 
the lower tail we have $t=O(\mu)$, it follows that 
$\sum_{k}pc_k^2 + t \max_k c_k = \Theta(\mu^2/(n^2p))$. 
For the decreasing function $f=-Y_H$ a combination of 
\eqref{McDExtVMon:Pr} and \eqref{McDExtV:Pr} now yields 
\[
\PP(Y_H \leq \mu -t ) \le \exp\left(-\Theta\left(\frac{t^2}{\mu^2/(n^2p)}\right)\right) ,
\]
which asymptotically matches~\eqref{eq:JansonEx}, i.e., the estimate 
of Janson's inequality. In fact, for $t=\Theta(\mu)$ this bound is 
best possible (up to constants in the exponent) since $G_{n,p}$ 
contains no edges (and thus no copies of $H$) with probability 
$e^{-\Theta(n^2p)}$.

\subsubsection{Application: the reverse $H$-free process}\label{sec:HfreeIntro}
The following variations of the classical random graph processes were 
proposed by Bollob{\'a}s and Erd{\H o}s at the 1990 Quo Vadis, Graph 
Theory conference in an attempt to improve Ramsey numbers~\cite{Bollobas2012PC,BollobasRiordan2008}. 
The \emph{$H$-free process}, where, starting with an empty graph on 
$n$ vertices, in each step a new edge is added, chosen uniformly at 
random from all pairs whose addition does not complete a copy of $H$. 
The \emph{reverse $H$-free process}, where, starting with a complete 
graph on $n$ vertices, in each step an edge is removed, chosen 
uniformly at random from all edges that are contained in a copy of $H$. 
The \emph{$H$-removal process}, where, starting with a complete graph 
on $n$ vertices, in each step all $e_H$ edges of a copy of $H$ are 
removed, which is selected uniformly at random from all $H$ copies. 
All of these processes end with an $H$-free graph, and Bollob{\'a}s 
and Erd{\H o}s asked (among other structural properties) what their 
typical final number of edges is~\cite{Bollobas2012PC,BollobasRiordan2008}.

These variations have received considerable attention in recent years, 
in particular the $H$-free process. Its typical final number of edges 
is nowadays known up to logarithmic factors~\cite{BohmanKeevash2010H,OsthusTaraz2001} 
for the class of \emph{strictly $2$-balanced} graphs $H$, where in 
\eqref{eq:d2} the inequality is strict. Matching bounds up to constant 
factors have only been established for some special forbidden graphs 
and the class of $C_{\ell}$-free processes, see e.g.~\cite{Warnke2010Cl,Warnke2010K4}. 
The final graph of the $K_s$-free process also yields the best known 
lower bounds on the Ramsey numbers $R(s,t)$ with $s \geq 4$, 
see~\cite{Bohman2009K3,BohmanKeevash2010H}. 
Recently Makai~\cite{Makai2012} determined the (asymptotic) final 
number of edges of the reverse $H$-free process for the class of 
strictly $2$-balanced graphs, but its final graph yields no new 
estimates for $R(s,t)$. 
Although the related $H$-removal process has been studied in several 
papers the final number of edges is known up to multiplicative 
$n^{o(1)}$ factors only in the special case $H=K_3$, see 
e.g.~\cite{BohmanFriezeLubetzky2012,Roedlthoma1996,Spencer1995}.

Using our typical bounded differences inequality, in 
Section~\ref{sec:Hfree} we show that the final number of edges in the 
reverse $H$-free process is sharply concentrated when $H$ is 
$2$-balanced (we do not assume \emph{strictly} $2$-balanced), and 
also determine the likely number of edges up to constants. 
This is in contrast to all known results for the widely studied 
$H$-free and $H$-removal processes. Indeed, in these 
(a) no sharp concentration results are known, 
(b) the order of magnitude of the final number of edges is open for 
most strictly $2$-balanced graphs, and 
(c) no general results apply to the class of $2$-balanced graphs. 
As we shall see, when $H$ is a matching the expected final number of 
edges in the reverse $H$-free process is $\Theta(1)$. 
When it comes to concentration we thus restrict our main attention to 
all other $2$-balanced graphs $H$, which in fact satisfy 
$d_2(H) \geq 1$ (with equality for trees). Here our next result shows 
that the reverse $H$-free process typically ends with 
$\Theta(n^{2-1/d_2(H)})$ edges, answering (up to constant factors) 
the aforementioned question of Bollob{\'a}s and Erd{\H o}s from 1990. 
\begin{theorem}\label{thm:revH}
Let $H$ be a $2$-balanced graph. There are constants $a,A>0$ such 
that the final number of edges $M_n$ in the reverse $H$-free process 
has expectation satisfying $a n^{2-1/d_2(H)} \leq \EE M_n \leq A n^{2-1/d_2(H)}$. 
Furthermore, for any $c > 0$ we have $|M_n-\EE M_n| \leq \sqrt{\EE M_n} (\log n)^{4 e_H}$ 
with probability at least $1-n^{-c}$ for $n \geq n_0(c,H)$. 
\end{theorem}
Our arguments partially generalize to arbitrary graphs. Set 
$d_2(K_2)=1/2$ and $m_2(H)=\max_{G \subseteq H,e_G \geq 1}d_2(G)$, so 
that $m_2(H)=d_2(H)$ for $2$-balanced graphs $H$. We show that for 
any graph the expected final number of edges in the reverse $H$-free 
process is $\Theta(n^{2-1/m_2(H)})$, and prove concentration under 
certain conditions (satisfied e.g.\ by a clique $K_r$ with an extra 
edge hanging off), see Section~\ref{sec:Hfree}. 
The proof of Theorem~\ref{thm:revH} also extends to a finite family 
of forbidden graphs $\cH$, which for the $H$-free process was 
considered in~\cite{OsthusTaraz2001}. Indeed, defining the reverse 
$\cH$-free process in the obvious way (always removing a random edge 
that is contained in a copy of some $H \in \cH$) we obtain, for 
example, the following generalization. 
\begin{theorem}\label{thm:revFH}
Let $\cH=\{H_1, \ldots, H_r\}$ be a family of $2$-balanced graphs. 
Define $H,J \in \cH$ such that $d_2(H) = \min_{F \in \cH}d_2(F)$ and 
$e_J = \max_{F \in \cH}e_F$. There are constants $a,A>0$ such that 
the final number of edges $M_n$ in the reverse $\cH$-free process has 
expectation satisfying $a n^{2-1/d_2(H)} \leq \EE M_n \leq A n^{2-1/d_2(H)}$. 
Furthermore, for any $c > 0$ we have $|M_n-\EE M_n| \leq \sqrt{\EE M_n} (\log n)^{4 e_J}$ 
with probability at least $1-n^{-c}$ for $n \geq n_0(c,\cH)$. 
\end{theorem}

\subsection{Organization of the paper}
Section~\ref{sec:proof} is devoted to the proof of our new 
concentration inequalities, which are then illustrated by an 
application to the $H$-free process in Section~\ref{sec:Hfree}.

\section{Proofs of the concentration inequalities}\label{sec:proof} 
We start by proving two general martingale inequalities. These are 
applied in Section~\ref{sec:BDI}, where we establish our variants of 
the bounded differences inequality.

\subsection{Martingale inequalities}\label{sec:Mart}
Our concentration results are based on the following variants of 
Hoeffding--Azuma/Bernstein-type martingale inequalities. Since they 
are not stated exactly in this form in the literature, we give short 
proofs for the readers convenience (following the slick approach of 
Freedman~\cite{Freedman1975}). In both we assume that 
$(\cF_k)_{0 \leq k \leq N}$ is an increasing sequence of 
$\sigma$-algebras, and $(M_k)_{0 \leq k \leq N}$ is an 
$(\cF_k)_{0 \leq k \leq N}$-adapted bounded martingale. 
\begin{lemma}[`Bounded differences martingale inequality']\label{lem:MAH}
Let $L_k$ and $U_k$ be $\cF_{k-1}$-measurable variables satisfying 
$L_k \leq M_k-M_{k-1} \leq U_k$. Set $S_k = \sum_{i \in [k]} (U_i-L_i)^2$. 
For every $t \ge 0$ and $S > 0$ we have 
\begin{equation}\label{eq:MAH}
\PP(\text{$M_k \ge M_0+t$ and $S_k \le S$ for some $k \in [N]$}) \le e^{-2t^2/S} .
\end{equation}
\end{lemma}
\begin{lemma}[`Bounded variances martingale inequality']\label{lem:MFr}
Let $U_k$ be an $\cF_{k-1}$-measurable variable satisfying 
$M_k-M_{k-1} \leq U_k$. Set $C_k = \max_{i \in [k]} U_i$ and 
$V_k = \sum_{i \in [k]} \Var(M_i-M_{i-1}\mid \cF_{i-1})$. Let 
$\phi(x)=(1+x)\log(1+x)-x$. For every $t \ge 0$ and $V,C > 0$ we have 
\begin{equation}\label{eq:MFr}
\PP(\text{$M_k \ge M_0+t$, $V_k \le V$ and $C_k \le C$ for some $k \in [N]$}) \le e^{-V/C^2 \cdot \phi(Ct/V)} \le e^{-t^2/(2V+2Ct/3)} . 
\end{equation}
\end{lemma}
\begin{remark}
Note that $V_k$ generalizes $S_k$ since 
$\Var(M_i-M_{i-1}\mid \cF_{i-1})=\EE((M_i-M_{i-1})^2\mid \cF_{i-1})$ 
holds (it is not hard to check that $V_k \leq S_k/4$). 
In fact, Lemmas~\ref{lem:MAH} and~\ref{lem:MFr} extend with minor 
modifications to supermartingales: defining 
$V_k = \sum_{i \in [k]} \EE((M_i-M_{i-1})^2\mid \cF_{i-1})$ suffices. 
\end{remark}
Observe that we allow for (accumulative) random bounds on the 
one-step changes (and other quantities), which in case of 
Lemma~\ref{lem:MAH} is the main difference to the usual formulation 
of the classical Hoeffding--Azuma inequality~\cite{Hoeffding1963,Azuma1967}. 
Lemma~\ref{lem:MFr} also extends the related Theorem~2.2.2 of Kim and 
Vu~\cite{KimVu2000} (see also Lemma~3.1 in Vu's survey~\cite{Vu2002}), 
which assumes that the underlying probability space is generated by 
independent random variables (of a special form).

Note that $L_k$, $U_k$ are $\cF_{k-1}$-measurable, whereas 
$M_{k}-M_{k-1}$ is $\cF_{k}$-measurable. This difference sometimes 
causes subtle off-by-one errors. As pointed out by Oliver Riordan, 
for e.g.\ the estimate 
\begin{equation}\label{eq:MAHE}
\textstyle
\PP(M_N \ge M_0+t) \le e^{-t^2/2\sum_k c_{k}^2} + \eta 
\end{equation}
it does \emph{not} suffice if $\sum_{k} \PP(|M_k-M_{k-1}| > c_k) \leq \eta$, 
as claimed by Theorem~8.4 in~\cite{ChungLu2006}. The problem is that, 
conditional on $\cF_{k-1}$, in the next step it sometimes is 
\emph{always} possible for $|M_{k}-M_{k-1}| \leq c_k$ to fail 
(although this might be unlikely). With this in mind, we see that 
\eqref{eq:MAHE} holds e.g.\ if 
\[
\textstyle
\sum_{k} \PP(\text{it is possible, given $M_1,...,M_{k-1}$, that $|M_k-M_{k-1}| > c_k$}) \leq \eta . 
\]
In fact, assuming that $|M_k-M_{k-1}| \leq C_k$ always holds, the 
approach of \cite{ChalkerGodboleHitczenkoRadcliffRuehr1999,ShamirSpencer1987} 
e.g.\ implies \eqref{eq:MAHE} if 
\[
\textstyle
\sum_{k} (1+2C_k/t) \cdot \PP(|M_k-M_{k-1}| > c_k/4) \leq \eta .
\]

\subsubsection{Proof of Lemmas~\ref{lem:MAH} and~\ref{lem:MFr}}
Our proofs use the following (standard) inequalities due to 
Hoeffding~\cite{Hoeffding1963} and Steiger~\cite{Steiger1969}; they 
follow e.g.\ from the proofs of Lemmas~2.4, 2.6 and~2.8 in 
McDiarmid's survey~\cite{McDiarmid1998}. 
\begin{lemma}\label{lem:Exp}
Let $X$ be random variable with $\EE(X \mid \cF) = 0$. 
Let $L,U$ be $\cF$-measurable random variables. 
Set $g(x)=(e^x-1-x)/x^2$ for $x \neq 0$ and $g(0)=1/2$. 
For any $\lambda \ge 0$ the following holds: 
\begin{eqnarray}
\label{eq:H} 
L \leq X \leq U & \implies & \EE(e^{\lambda X} \mid \cF) \le e^{\lambda^2(U-L)^2/8} \quad \text{and}\\
\label{eq:St}
X \leq U & \implies & \EE(e^{\lambda X} \mid \cF) \le e^{\lambda^2g(\lambda U)\Var(X\mid\cF)} .
\end{eqnarray}
Furthermore, $g(x)$ is a non-negative increasing function. \nopf
\end{lemma}
\begin{lemma}\label{lem:phi}
Set $\phi(x)=(1+x)\log(1+x)-x$. For all $x \geq 0$ we have 
$\phi(x) \geq x^2/(2+2x/3)$. \nopf 
\end{lemma}
\begin{proof}[Proof of Lemmas~\ref{lem:MAH} and~\ref{lem:MFr}] 
Set
\[ 
W_k = \sum_{i \in [k]} g(\lambda U_i)\Var(M_i-M_{i-1} \mid \cF_{i-1}) . 
\] 
The key point is that $M_{k-1}$ and $L_k$, $U_k$, $S_k$, $W_k$ are 
$\cF_{k-1}$ measurable. So, by applying \eqref{eq:H} and 
\eqref{eq:St} to $\EE(e^{\lambda (M_k-M_{k-1})} \mid \cF_{k-1})$ we 
see that 
\[
Y_{k}=e^{\lambda (M_k-M_0)-\lambda^2 S_k/8} \quad \text{ and } \quad Z_{k}=e^{\lambda (M_k-M_0)-\lambda^2 W_k}
\]
satisfy $\EE(Y_k \mid \cF_{k-1}) \leq Y_{k-1}$ and 
$\EE(Z_k \mid \cF_{k-1}) \leq Z_{k-1}$, i.e., are supermartingales. 
We define the stopping time $T$ as the minimum of $N$ and the 
smallest $k \in [N]$ with $M_k-M_0 \ge t$; as usual, we write 
$i \wedge T$ as shorthand for $\min\{i,T\}$. By construction 
$(Y_{k \wedge T})_{0 \leq k \leq N}$ and 
$(Z_{k \wedge T})_{0 \leq k \leq N}$ are both supermartingales. 
In particular, we have 
\[
\EE Y_{N \wedge T} \leq \EE Y_0=1 \quad \text{ and } \quad \EE Z_{N \wedge T} \leq \EE Z_0=1 . 
\]

Let $\cE_N$ denote the event that $M_k\ge M_0 + t$ and $S_k \le S$ 
for some $k \in [N]$. Note that $\cE_N$ implies 
$Y_{N \wedge T}=Y_{T} \ge e^{\lambda t-\lambda^2 S/8}$. So, for 
$\lambda = 4t/S$ Markov's inequality gives 
\[
\PP(\cE_N) \le \PP(Y_{N \wedge T} \ge e^{\lambda t-\lambda^2 S/8}) \le e^{\lambda^2 S/8-\lambda t}=e^{-2t^2/S} , 
\]
which establishes \eqref{eq:MAH} and thus Lemma~\ref{lem:MAH}.

We proceed similarly for $(Z_{k \wedge T})_{0 \leq k \leq N}$ and let 
$\cE'_N$ denote the event that $M_k \ge M_0 + t$, $V_k \le V$ and 
$C_k \le C$ for some $k \in [N]$. Using $V_k \geq 0$ and monotonicity 
of $g(x) \geq 0$ we see that $\cE'_N$ implies 
$Z_{N \wedge T} \ge e^{\lambda t-\lambda^2 g(\lambda C)V_k} \geq e^{\lambda t-\lambda^2 g(\lambda C)V}$. 
Recall that $\phi(x)=(1+x)\log(1+x)-x$. For $\lambda=\log(1+Ct/V)/C^2$ 
Markov's inequality and Lemma~\ref{lem:phi} now yield 
\begin{equation*}
\PP(\cE'_N) \le e^{\lambda^2g(\lambda C) V-\lambda t}= e^{-V/C^2 \cdot \phi(Ct/V)} \le e^{-t^2/(2V+2Ct/3)} , 
\end{equation*}
which establishes \eqref{eq:MFr} and thus Lemma~\ref{lem:MFr}. 
\end{proof}

\subsection{Bounded differences inequalities}\label{sec:BDI}
The textbook proof of Theorem~\ref{thm:McD} is based on the 
Hoeffding--Azuma inequality~\cite{Hoeffding1963,Azuma1967}, and 
essentially uses the `worst case' Lipschitz condition 
\eqref{eq:fL} to apply Lemma~\ref{lem:MAH} with $|U_k-L_k| \leq c_k$. 
We need some modifications to deal with the obstacle that the `good' 
event $\Gamma$ and thus the `typical case' in \eqref{eq:fTL} does 
\emph{not} always hold, and these are partially inspired by the 
seminal work of Shamir and Spencer~\cite{ShamirSpencer1987} from 
1987.

When $\PP(X \notin \Gamma) \leq \eta$ holds one might be tempted to 
add $\eta$ to the error bound and then always assume that $\Gamma$ 
holds. The problem is that in the martingale based proof one needs 
to estimate \emph{conditional} expected changes as in \eqref{eq:fLE}. 
So, informally speaking, despite $\omega \in \Gamma$ the `good' 
event can still fail `inside' the corresponding expectations. 
One can try to overcome this by conditioning on $\Gamma$, but this 
usually introduces a new technical problem: then the variables 
are not conditionally independent (in which case Lipschitz 
conditions comparable to \eqref{eq:fTL} no longer suffice to bound 
the expected changes). These technicalities seem to cause some 
confusion in e.g.~\cite{DubhashiPanconesi2009,Grable1998}.

We step aside these issues by noting that for good bounds on 
conditional \emph{expected} one-step changes it suffices that the 
conditional probabilities of large changes are small. 
One key aspect of our approach is that we can always guarantee 
this via the `global' event $\Gamma$ only, i.e., \emph{without} 
having any knowledge about the corresponding conditional 
distributions.

\subsubsection{The general approach}\label{sec:PrMain}
We now introduce the setup used in all subsequent proofs. 
Let $Y=f(X)$. We consider the increasing sequence of 
sub-$\sigma$-fields $\cF_k$ generated by $X_1, \ldots, X_k$. Using 
Doob's construction, the sequence $Y_k=\EE(Y\mid \cF_k)$ is a 
martingale with $Y_0=\EE f(X)=\mu$ and $Y_N=f(X)$. Now we define 
$\cF_{k-1}$-measurable events $\cB_{k-1}$, where 
$\omega\in \cB_{k-1}$ if 
\begin{equation}\label{eq:Bad} 
\PP(X \notin \Gamma \mid \cF_{k-1})(\omega) > \gamma_{k} . 
\end{equation}
Let $\cB=\neg\Gamma \cup \bigcup_{k \in [N]} \cB_{k-1}$. Note that 
$\PP(X \notin \Gamma \mid \cF_{0}) = \PP(X \notin \Gamma)$ yields 
$\PP(\cB_0)=0$ if $\gamma_1 \geq \PP(X \notin \Gamma)$ and 
$\PP(\cB_0)=1$ otherwise. Using $\gamma_1 \in (0,1]$ we infer 
$\PP(\neg\Gamma \cup \cB_0)\leq \gamma_1^{-1}\PP(X \notin \Gamma)$. 
Observing that $\PP(X \notin \Gamma) = \EE (\PP(X \notin \Gamma \mid \cF_{k-1})) \geq \gamma_k \PP(\cB_{k-1})$, 
the union bound now gives 
\begin{equation}\label{eq:BadUB}
\PP( \cB) \le \PP(\neg\Gamma \cup \cB_0) + \sum_{2 \leq k \leq N} \PP(\cB_{k-1}) \leq \sum_{k \in [N]} \gamma_k^{-1} \cdot \PP(X \notin \Gamma) . 
\end{equation}
Let the stopping time $T$ be the minimum of $N$ and the smallest 
$0 \le k < N$ for which $\cB_k$ holds (note that $T \leq k-1$ is 
$\cF_{k-1}$-measurable). Setting $M_k=Y_{k \wedge T}$, it follows 
that the sequence $(M_{k})_{0 \leq k \leq N}$ is a martingale with 
$Y_0=M_0=\mu$. Since $T=N$ unless $\cB$ holds, recalling $Y_N=f(X)$ 
we see that 
\begin{equation}\label{eq:UB}
\PP(f(X) \ge \mu+t \text{ and } \neg \cB) = \PP(Y_N \ge Y_0+t \text{ and } \neg \cB) \leq \PP(M_N \ge M_0+t) . 
\end{equation}
It remains to establish suitable tail estimates for 
$\PP(M_N \ge M_0+t)$, and via Lemmas~\ref{lem:MAH} and~\ref{lem:MFr} 
this reduces to proving (deterministic) upper bounds on the random 
variables $S_N$, $V_N$ and $C_N$. To this end we consider the 
martingale difference sequences $\DM_{k}=M_k-M_{k-1}$ and 
$\DY_{k}=Y_k-Y_{k-1}$, which satisfy $\EE(\DM_{k}\mid \cF_{k-1})=0$ 
and $\EE(\DY_{k}\mid \cF_{k-1})=0$. Set $e_k=\gamma_k (d_{k}-c_{k})$ 
and $\Delta_k=c_{k}+e_k$.

\begin{proof}[Proof of Theorem~\ref{thm:McDExt}] 
It suffices to show $\DM_{k} \in [-\Delta_k,\Delta_k]$ for each 
$k \in [N]$: then the claim follows by applying Lemma~\ref{lem:MAH} 
with $S=\sum_{k \in [N]}(2\Delta_k)^2$. The following argument is 
written with an eye on the upcoming proofs (where some modifications 
are needed). Note that $\DM_{k}=0$ if $T \leq k-1$ and $\DM_k=\DY_{k}$ 
if $T \geq k$. So it is enough to prove that $|\DY_{k}| \le \Delta_k$ 
whenever $T \geq k$. For brevity, for $z \in \Lambda_k$ and 
$y=(y_{k+1}, \ldots, y_N) \in \prod_{k < j \le N} \Lambda_j$ we write 
$f_y(z)$ for $f(X_1, \ldots, X_{k-1},z, y_{k+1}, \ldots, y_N)$. Note 
that 
\begin{equation*}\label{eq:FXE}
\begin{split}
\EE ( f(X) \mid \cF_{k-1}, X_k=a) &= \sum_{y_{k+1}, \ldots, y_N} f_y(a) \; \PP(X_{k+1}=y_{k+1},\ldots,X_{N}=y_{N} \mid \cF_{k-1}, X_k=a) . 
\end{split}
\end{equation*}
Defining $|\DY_{k}(a,b)|$ via the next equation, since 
$X_1, \ldots, X_N$ are independent it follows that 
\begin{equation}\label{eq:MBCI}
\begin{split}
|\DY_{k}(a,b)| &= |\EE ( f(X) \mid \cF_{k-1},X_k=a)-\EE ( f(X) \mid \cF_{k-1},X_k=b)|\\
& \leq \sum_{y_{k+1}, \ldots, y_N} |f_y(a)-f_y(b)| \; \PP(X_{k+1}=y_{k+1},\ldots,X_{N}=y_{N} \mid \cF_{k-1},X_k=a) . 
\end{split}
\end{equation}
By distinguishing between $X \in \Gamma$ and $X \notin \Gamma$, each 
time applying \eqref{eq:fTL} as appropriate, we infer 
\begin{equation}\label{eq:MBCI2}
\begin{split}
|\DY_{k}(a,b)| & \leq c_k \PP(X \in \Gamma \mid \cF_{k-1},X_k=a) + d_k \PP(X \not\in \Gamma\mid \cF_{k-1},X_k=a)\\
& = c_k + (d_k-c_k) \PP(X \not\in \Gamma \mid \cF_{k-1},X_k=a) . 
\end{split}
\end{equation}
Recall that $\cB_{k-1}$ fails if $T \geq k$. Using 
$\sum_{y_k}\PP(X_k=y_k \mid \cF_{k-1})=1$ together with 
\eqref{eq:MBCI2} and \eqref{eq:Bad}, for $T \geq k$ we deduce 
\begin{equation}\label{eq:MBC}
\begin{split}
|\DY_{k}| &= |\EE ( f(X) \mid \cF_{k-1})-\EE ( f(X) \mid \cF_k)| \leq \sum_{y_{k}} |\DY_{k}(y_k,X_k)| \; \PP(X_{k}=y_{k} \mid \cF_{k-1})\\
& \leq c_k + (d_k-c_k) \PP(X \not\in \Gamma \mid \cF_{k-1}) \leq c_k + \gamma_k (d_k-c_k) = \Delta_k . 
\end{split}
\end{equation}
As explained, this completes the proof. 
\end{proof}
Here we could have used the classical Hoeffding--Azuma 
inequality~\cite{Hoeffding1963,Azuma1967} since the proof yields 
(deterministic) bounds for each individual $\DM_{k}$. 
We decided to apply Lemma~\ref{lem:MAH} since the forthcoming 
modifications needed for the `dynamic exposure' of Section~\ref{sec:AQ} 
do use its full strength, i.e., that accumulative estimates of the 
$\DM_{k}$ suffice.

\begin{proof}[Proof of Remark~\ref{rem:McDExt}]
Following the approach of McDiarmid~\cite{McDiarmid1989} we now modify 
the proof of Theorem~\ref{thm:McDExt} whenever $X_k$ takes only two 
values, say, $\Lambda_k=\{0,1\}$. We focus on the relevant case 
$T \geq k$, where $\DM_k=\DY_{k}$. Define $L_{k}$ and $U_{k}$ as the 
minimum and maximum of $\EE ( f(X) \mid \cF_{k-1},X_k=z)- \EE ( f(X) \mid \cF_{k-1})$ 
for $z \in \{0,1\}$. Clearly $L_k$ and $U_k$ are $\cF_{k-1}$-measurable 
and satisfy $L_k \leq \DY_{k} \leq U_k$. The key observation is that, 
using $T \ge k$, there exists an $\cF_{k-1}$-measurable 
$\alpha \in \{0,1\}$ satisfying 
\begin{equation}\label{eq:gammaK}
\gamma_k \geq \PP(X \not\in \Gamma \mid \cF_{k-1}) \geq \PP(X \not\in \Gamma \mid \cF_{k-1},X_k=\alpha) . 
\end{equation} 
So, since $X_k \in \{0,1\}$ takes only two values, using 
\eqref{eq:MBCI2} and \eqref{eq:gammaK} we infer for $T \geq k$ that 
\begin{equation}\label{eq:MBC3}
|U_k-L_k| \leq |\DY_{k}(\alpha,1-\alpha)| \leq c_k + (d_k-c_k) \PP(X \not\in \Gamma \mid \cF_{k-1},X_k=\alpha) \leq \Delta_k . 
\end{equation}
This completes the proof (by applying Lemma~\ref{lem:MAH} with 
$S=\sum_{k \in [N]}\Delta_k^2$). 
\end{proof}
In fact, Theorem~\ref{thm:McD} follows by a similar modification 
(here \eqref{eq:MBCI2} implies $\max_{a,b}|\DY_{k}(a,b)| \leq c_k$).

\begin{proof}[Proof of Theorem~\ref{thm:McDExtV}]
In the proof of Theorem~\ref{thm:McDExt} we established 
$\DM_{k} \leq \Delta_k$ for every $k \in [N]$. In view of this it 
suffices to show $\Var(\DM_k\mid\cF_{k-1})\le (1-p_k)p_k\Delta_k^2$ 
for each $k \in [N]$: then the claim follows by applying 
Lemma~\ref{lem:MFr} with $V=\sum_{k \in [N]}(1-p_k)p_k\Delta_k^2$ and 
$C=\max_{k \in [N]}\Delta_k$. Observe that $\EE(\DM_k \mid \cF_{k-1})=0$ 
implies $\Var(\DM_k\mid\cF_{k-1})=\EE(\DM_k^2\mid\cF_{k-1})$. Recall 
that $\DM_{k}=0$ if $T \leq k-1$ and $\DM_{k}=\DY_{k}$ if $T \geq k$. 
Combining these facts it is enough to prove that 
$\EE(\DY_{k}^2\mid\cF_{k-1}) \le (1-p_k)p_k\Delta_k^2$ whenever 
$T \geq k$. Set $D_k=\EE(Y \mid \cF_{k-1},X_k=1)-\EE(Y \mid \cF_{k-1},X_k=0)$. 
Recalling $\DY_k=\EE(Y\mid \cF_k)-\EE(Y\mid \cF_{k-1})$ we see that 
\begin{equation}\label{eq:DYk:two}
|\DY_k| \leq |D_k| \sum_{\beta \in \{0,1\}} \PP(X_k=\beta \mid \cF_{k-1}) \indic{X_k = 1-\beta} . 
\end{equation} 
Arguing as in~\eqref{eq:gammaK} and \eqref{eq:MBC3} we readily obtain 
$|D_k| \leq \Delta_k$ when $T \geq k$, and thus infer 
\[
\DY_k^2 \leq \Delta_k^2 \sum_{\beta \in \{0,1\}} \PP(X_k=\beta \mid \cF_{k-1})^2 \indic{X_k = 1-\beta} . 
\]
Using the independence of $X_1, \ldots, X_N$ it follows that for 
$T \ge k$ we have 
\begin{equation}\label{eq:EDYk:two}
\EE(\DY_{k}^2\mid\cF_{k-1}) \leq \Delta_k^2 \sum_{\beta \in \{0,1\}} \PP(X_k=\beta)^2 \PP(X_k = 1-\beta) = (1-p_k)p_k \Delta_k^2 , 
\end{equation}
where we used $\PP(X_k=1)=p_k$ and $(1-x)^2x+x^2(1-x)= (1-x)x$ for 
the last inequality. 
\end{proof}
Note that \eqref{eq:DYk:two} implies 
$\DM_{k} \leq \max\{1-p_k,p_k\} \cdot \Delta_k$, but the resulting 
minor improvement of $C$ usually has negligible effect.

\begin{proof}[Proof of Remark~\ref{rem:McDExtV}]
In the more general situation where each $X_k$ takes values in a set 
$\Lambda_k$ and satisfies $\max_{\eta \in \Lambda_k} \PP(X_k=\eta) \ge 1-p_k$,
we first show that \eqref{McDExtV:Pr} holds after deleting $(1-p_k)$ 
and replacing $c_k+e_k$ with $\tilde{\Delta}_k=c_k+e_k \cdot (1-p_k)^{-1}$. 
With the proof of Theorem~\ref{thm:McDExtV} in mind it suffices to show 
$\Var(\DY_{k}\mid\cF_{k-1}) \le p_k\tilde{\Delta}_k^2$ whenever $T \ge k$. 
For $\beta \in \Lambda_k$ satisfying $\PP(X_k=\beta) \geq 1-p_k$ set 
$\tau_{k}=\EE(Y \mid \cF_{k-1},X_k=\beta)$ and $D_k=\EE(Y\mid \cF_k) -\tau_{k}$. 
Note that, using the independence of $X_1, \ldots, X_N$, we have 
$\PP(D_k \neq 0 \mid \cF_{k-1}) \le \PP(X_k \neq \beta) \le p_k$. 
We claim that it suffices to show $|D_k| \leq \tilde{\Delta}_k$. 
Indeed, since $Y_{k-1}$ and $\tau_{k}$ are $\cF_{k-1}$ measurable, we 
have 
\[
\Var(\DY_{k} \mid \cF_{k-1})=\Var(D_k \mid \cF_{k-1}) \leq \EE(D_k^2 \mid \cF_{k-1}) \leq \tilde{\Delta}_k^2 \PP(D_k \neq 0 \mid \cF_{k-1}) \leq p_k \tilde{\Delta}_k^2 . 
\]
To bound $|D_k|$, first note that $T \ge k$ and independence of 
$X_1, \ldots, X_N$ yields 
\begin{equation}\label{eq:gammaK2}
\gamma_k \geq \PP(X \not\in \Gamma \mid \cF_{k-1}) \geq \PP(X \not\in \Gamma \mid \cF_{k-1},X_k=\beta) (1-p_k) . 
\end{equation}
So, using \eqref{eq:MBCI2} and \eqref{eq:gammaK2}, for $T \geq k$ we 
infer 
\begin{equation}\label{eq:DkG}
\begin{split}
|D_k| &= |\EE(Y \mid \cF_{k-1},X_k=\beta)-\EE(Y \mid \cF_{k})| = |\DY_k(\beta,X_k)|\\
&\leq c_k + (d_k-c_k) \PP(X \not\in \Gamma \mid \cF_{k-1},X_k=\beta) \leq c_k+ \gamma_k(1-p_k)^{-1} \cdot (d_k-c_k) = \tilde{\Delta}_k , 
\end{split}
\end{equation}
establishing the claim.

A similar argument shows that \eqref{McDExtVCor:Pr} holds after 
deleting $(1-p_k)$. The point is that in Corollary~\ref{Cor:McDExtVCor} 
there is no `good' event $\Gamma$. Consequently, when invoking 
\eqref{eq:MBCI2} in \eqref{eq:DkG} the standard line of reasoning 
(using \eqref{eq:fL} instead of \eqref{eq:fTL}) yields 
$|D_k| \le c_k$, and the claim follows. 
\end{proof}

\begin{proof}[Proof of Theorem~\ref{thm:McDfTLP}]
The crux is that \eqref{eq:MBCI2} and \eqref{eq:MBC} are at the heart 
of all previous proofs. In the following we exploit that both can be 
adapted when \eqref{eq:fTLP} instead of \eqref{eq:fTL} holds: it 
suffices if $e_k=2\gamma_k (d_{k}-c_{k})q_k^{-1}$ is used, 
where $\min_{\eta \in \Lambda_k} \PP(X_k=\eta) \geq q_k$.

We start by modifying the proof of Theorem~\ref{thm:McDExt}. 
Analogous to \eqref{eq:gammaK2}, if $T \geq k$ then for all 
$b \in \Lambda_k$ we have 
\begin{equation*}\label{eq:gammaKM}
\gamma_k \ge \PP(X \not\in \Gamma \mid \cF_{k-1}) \geq \PP(X \not\in \Gamma \mid \cF_{k-1},X_k=b) q_k . 
\end{equation*}
The key point of \eqref{eq:fTLP} is that $|f(x)-f(\tilde{x})| \leq c_k$ 
only holds if $x,\tilde{x} \in \Gamma$. So, using \eqref{eq:fTLP} as 
appropriate, the corresponding variant of \eqref{eq:MBCI2} for 
$T \ge k$ is 
\begin{equation}\label{eq:MBCIM}
\begin{split}
|\DY_{k}(a,b)| & \leq c_k + (d_k-c_k) \left[\PP(X \not\in \Gamma \mid \cF_{k-1},X_k=a)+\PP(X \not\in \Gamma \mid \cF_{k-1},X_k=b)\right] \\
& \leq c_k + (d_k-c_k) \left[\PP(X \not\in \Gamma \mid \cF_{k-1},X_k=a)+\gamma_k q_k^{-1} \right] . 
\end{split}
\end{equation}
Now, arguing as in \eqref{eq:MBC} and using $1+q_k^{-1} \le 2q_k^{-1}$, 
we obtain a natural analogue for $T \ge k$, namely 
\begin{equation}\label{eq:MBCM}
\begin{split}
|\DY_{k}| & \leq c_k + (d_k-c_k) \left[\PP(X \not\in \Gamma \mid \cF_{k-1}) +\gamma_k q_k^{-1} \right] \leq c_k + 2\gamma_k (d_k-c_k)q_k^{-1} = \Delta_k , 
\end{split}
\end{equation}
which establishes the claimed variant of Theorem~\ref{thm:McDExt}.

In the proofs of Remark~\ref{rem:McDExt} and Theorem~\ref{thm:McDExtV} 
we only need to adapt \eqref{eq:MBC3}, and using \eqref{eq:MBCIM} this 
follows by straightforward modifications. Similarly, in the proof of 
Remark~\ref{rem:McDExtV} it suffices to modify \eqref{eq:DkG}, which 
is standard using \eqref{eq:MBCM} together with 
$(1-p_k)^{-1}+q_k^{-1} \leq 2q_k^{-1}(1-p_k)^{-1}$. 
\end{proof}

\subsubsection{Some extensions}\label{sec:PrExt}
\begin{proof}[Proof of Remark~\ref{rem:MON}]
Note that $f(X) \geq \mu +t$ is increasing (decreasing) if $f(X)$ is 
increasing (decreasing). Furthermore, in view of \eqref{eq:Bad} it is 
easy to check that $\cB_{k-1}$ is increasing (decreasing) if $\Gamma$ 
is decreasing (increasing). Using the definition $\cB$ and the 
assumptions of Remark~\ref{rem:MON}, it follows that $f(X) \ge \mu+t$ 
and $\neg\cB$ are either both increasing or decreasing. 
So Harris' inequality~\cite{Harris1960} yields 
\[
\PP(f(X) \ge \mu+t \text{ and } \neg\cB) \geq \PP(f(X) \ge \mu+t) \cdot \PP(\neg \cB) , 
\]
which readily establishes \eqref{McDExtVMon:Pr}. 
\end{proof}

\begin{proof}[Proof of Theorem~\ref{thm:McDExtP}] 
The basic idea is to use a truncation that maps every 
$x_k \notin \Gamma_k$ to some fixed $z_k \in \Gamma_k$. As before, we 
work with the sub-$\sigma$-fields $\cF_k$ generated by 
$X_1, \ldots, X_k$. Recall that $\cB$ is defined via \eqref{eq:Bad} 
and satisfies $\neg \cB \subseteq \Gamma$. Since 
$\PP(f(X) \ge \mu+t \text{ and } \neg\cB) \leq \PP(X \in \Gamma)$ we 
may assume that $\PP(X \in \Gamma) > 0$ and fix some 
$z=(z_1, \ldots, z_N) \in \Gamma \subseteq \prod_{j \in [N]}\Gamma_j$. 
For $x=(x_1, \ldots, x_N)$ we now define $x^*=(x_1^*, \ldots, x_N^*)$ 
via 
\begin{equation}\label{eq:MapX}
x_k^* = \begin{cases}
		x_k & \;\text{if $x_k \in \Gamma_k$,}\\
		z_k & \;\text{if $x_k \not\in \Gamma_k$.}
	\end{cases} 
\end{equation}
The key properties of this construction are (a) that $x \in \Gamma$ 
implies $x^*=x$, and (b) that $X^*=(X_1^{*}, \ldots, X_N^{*})$ is a 
family of independent random variables. Set $\mu^*=\EE f(X^*)$. 
We have $|\mu-\mu^*| \leq \EE |f(X)-f(X^*)| \leq s \PP(X \notin \Gamma)=\Delta$ 
(this is not best possible but keeps the formulas simple), and so 
$\Gamma \subseteq \neg \cB$ yields 
\begin{equation}\label{eq:PrgX}
\PP(f(X) \ge \mu+t+\Delta \text{ and } \neg\cB) \le \PP(f(X^*) \ge \mu^*+t \text{ and } \neg\cB) . 
\end{equation}
Now we estimate the right hand side of \eqref{eq:PrgX} for $Y=f(X^*)$ 
as in the proof of Theorem~\ref{thm:McDExt} via~\eqref{eq:UB}, and 
there are only two minor differences. 
The first is that due to the projection \eqref{eq:MapX} we have 
$X^* \in \prod_{j \in [N]}\Gamma_j$, so that the `refined' Lipschitz 
coefficients $c_k$ and $d_k$ of Theorem~\ref{thm:McDExtP} always apply. 
The second concerns the case distinction $X^* \in \Gamma$ and 
$X^* \notin \Gamma$. Here we use that $X^* \not\in \Gamma$ implies 
$X \not\in \Gamma$ pointwise, which yields $\PP(X^* \notin \Gamma \mid \cF_{k-1}) \leq \PP(X \notin \Gamma \mid \cF_{k-1})$. 
With this estimate the conclusion of \eqref{eq:MBC} readily carries 
over, completing the proof. 
\end{proof}
Here we could have estimated $f(X^*) \ge \mu^*+t$ via 
Theorem~\ref{thm:McDExt} (with $\Lambda_k$ replaced by $\Gamma_k$), 
using that $\PP(X^*\notin \Gamma) \leq \PP(X \notin \Gamma)$. The 
advantage of our more pedestrian approach is that it uses the same 
`bad' event $\cB$ as all other proofs (by applying 
Theorem~\ref{thm:McDExt} it would depend on $z$). 
\begin{proof}[Proof of Remark~\ref{rem:McDExtP}]
The monotonicity property implies $\mu \geq \mu^*$ (writing 
$f(X)-f(X^*)$ as a difference sequence of coordinate changes), so 
$\Delta=0$ suffices using $\mu+t \geq \mu^*+t$ in \eqref{eq:PrgX}. 
Turning to the special case $\Gamma = \prod_{j \in [N]}\Gamma_j$, 
note that $\cB=\neg \Gamma$ suffices to establish \eqref{eq:PrgX}. 
Now, since $X^*=(X_1^{*}, \ldots, X_N^{*})$ is a family of 
independent random variables with $X_k^{*} \in \Gamma_k$ satisfying 
(L), the claimed variant readily follows from Theorem~\ref{thm:McD}. 
\end{proof}

\subsubsection{Variants using dynamic exposure}\label{sec:PrDI}
In the following we briefly sketch how to modify the proofs of 
Sections~\ref{sec:PrMain} and~\ref{sec:PrExt} in case the variables 
are exposed in a dynamic order, which will eventually establish 
Theorem~\ref{thm:DynImpr} and~\ref{thm:DynImpr2} (in contrast 
to~\cite{AlonKimSpencer1997} our approach is based on general 
martingale inequalities). Recall that the strategies introduced in 
Section~\ref{sec:AQ} sequentially expose $X_{q_1}, X_{q_2}, \ldots$ 
with $q_i=q_i(X_{q_1}, \ldots, X_{q_{i-1}})$, where $f(X)$ is 
determined by $(X_1, \ldots, X_{q_k})$ with $k < N$ if $q_{k+1}=q_k$. 
For technical reasons we slightly modify these strategies so that 
(always) all variables are queried. More precisely, for the proof of 
Theorem~\ref{thm:DynImpr} we set $\tilde{q}_{k}=q_k$ until $f(X)$ is 
determined by $(X_1, \ldots, X_{q_k})$; afterwards 
$\tilde{q}_{k+1}, \ldots, \tilde{q}_{N}$ equals the remaining 
`useless' indices $[N] \setminus \{q_1, \ldots, q_k\}$ in ascending 
order, say (for the proof of Theorem~\ref{thm:DynImpr2} we simply use 
the fixed order $\tilde{q}_k=k$ for all $k \in [N]$). 
We consider an increasing sequence of sub-$\sigma$-fields, where 
$\cF_k$ is generated by $X_{\tilde{q}_1}, \ldots, X_{\tilde{q}_k}$. 
Note that each index $\tilde{q}_k$ is $\cF_{k-1}$-measurable. 
Furthermore, our modification ensures that the following two key 
properties hold: $X$ is $\cF_N$-measurable (this is needed to apply 
$\Gamma$ since the value of $f(X)$ must not uniquely determine $X$), 
and conditional on $\cF_{k-1}$ all $X_j$ with 
$j \notin \{\tilde{q}_1, \ldots, \tilde{q}_{k-1}\}$ are independent 
random variables. Define $R=\max_{Q \in \cQ}|Q|$ and 
$\cB=\neg\Gamma \cup \bigcup_{k \in [R]} \cB_{k-1}$ (in case of the 
fixed order $\tilde{q}_k=k$ we set $R=N$, so $\cB$ remains 
unchanged). Since the definition of $\cB_{k-1}$ via \eqref{eq:Bad} 
involves $\cF_{k-1}$, it follows that $\cB$ depends on the query 
strategy (unless the fixed order $\tilde{q}_k=k$ is used, as in the 
proof of Theorem~\ref{thm:DynImpr2}).

With these changes in mind, all arguments of Sections~\ref{sec:PrMain} 
and~\ref{sec:PrExt} essentially carry over word by word, the only 
exception being the proof of Remark~\ref{rem:MON} (the monotonicity 
argument needs a fixed order such as $\tilde{q}_k=k$). 
The crucial observation is that for every variable $X_\ell$ not 
queried by the original strategy we know that its value will 
\emph{not} change the outcome of $f(X)$. To be more formal, the key 
point is that whenever such a `useless' variable is queried in step 
$i$ we have $\DM_i=0$ (note that for $i>R$ this is always the case), 
i.e., the indices of these variables do not contribute to $S_N$, 
$V_N$ or $C_N$. 
Observe that due to the dynamic exposure we `only' have a connection 
between $\gamma_k$ and the index $\tilde{q}_k$. We overcome this 
minor complication using the assumption that $\gamma_k=\gamma$ for 
all $k \in [N]$ (we may allow for different $\gamma_k$ if 
$\tilde{q}_k=k$ is used), which also ensures that 
$\max_{Q \in \cQ}\sum_{k \in Q} \gamma_k^{-1}=\sum_{k \in [R]}\gamma_k^{-1}$ 
holds (in case of $\tilde{q}_k=k$ the estimate \eqref{eq:BadUB} stays 
unchanged).

The remaining details for establishing Theorem~\ref{thm:DynImpr} 
and~\ref{thm:DynImpr2} are rather straightforward: when invoking the 
martingale estimates we simply take the `worst case' bounds for $S$, 
$V$ and $C$ over all possible sets of queried indices $Q \in \cQ$ 
(where $Q$ and $\cQ$ are as defined in Section~\ref{sec:AQ}); for 
example, using $S=\max_{Q \in \cQ}\sum_{k \in Q}(2\Delta_k)^2$ in 
case of Theorem~\ref{thm:McDExt}. It is this last step where the 
accumulative random bounds in Lemmas~\ref{lem:MAH} and \ref{lem:MFr} 
are crucial (the behaviour of each individual $\DM_k$ may vary 
significantly for different sample points due to the dynamic order in 
which the variables are queried).

\subsubsection{Variants using the general Lipschitz condition}\label{sec:PrGV}
Finally, we discuss how to modify the proofs in Section~\ref{sec:PrMain} 
when the independence assumption is replaced by (GL). 
We first claim that $\rho_k=\rho_k(\Sigma_a,\Sigma_b):\Sigma_a \to \Sigma_b$ 
is a bijection with equality in \eqref{eq:PGTL}, i.e., satisfies 
\begin{equation}\label{eq:PGTLE}
\PP(X=x \mid X \in \Sigma_a) = \PP(X=\rho_k(x) \mid X \in \Sigma_b) .
\end{equation}
Indeed, using \eqref{eq:PGTL} and that $\rho_k$ is injective it 
follows that 
\[
\sum_{x \in \Sigma_a} \PP(X=x \mid X \in \Sigma_a) \leq \sum_{x \in \Sigma_a} \PP(X=\rho_k(x) \mid X \in \Sigma_b) \leq \sum_{x \in \Sigma_b} \PP(X=x \mid X \in \Sigma_b) . 
\]
Noting that $\sum_{x \in \Sigma_z} \PP(X=x \mid X \in \Sigma_z)=1$ 
for $z \in \{a,b\}$ with $|\Sigma_z|>0$ we infer that all 
inequalities are in fact equalities, which establishes 
\eqref{eq:PGTLE}. Since every $x \in \Sigma_z$ satisfies $\PP(X=x)>0$ 
it also follows that $\rho_k$ must be a bijection, as claimed.

In preparation of our forthcoming arguments we now relate the 
definitions used in the proofs of Section~\ref{sec:PrMain} with those 
occurring in (GL). Analogous to $\Sigma_z$, given any possible 
sequence of outcomes $a_1, \ldots, a_{k-1}$ of $X_1, \ldots, X_{k-1}$ 
define $\Sigma$ as the set of all 
$x=(a_1, \ldots, a_{k-1},x_k, \ldots, x_N) \in \prod_{j \in [N]}\Lambda_j$ 
with $\PP(X=x)>0$. Recall that $\cF_k$ is the increasing sequence of 
sub-$\sigma$-fields generated by $X_1, \ldots, X_k$. The key point is 
that $(\cF_k)_{0 \leq k \leq N}$ naturally corresponds to an 
increasing sequence of partitions of the sample space, where two 
points belong to the same part if and only if they agree on the first 
$k$ coordinates. For example, for $\omega \in \Sigma$ we have 
\begin{equation}
\label{eq:Conv}
\EE(\cdot \mid \cF_{k-1},X_k=z)(\omega)=\EE(\cdot \mid X \in \Sigma_z)
\quad \text{ and } \quad 
\PP(\cdot \mid \cF_{k-1},X_k=z)(\omega)=\PP(\cdot \mid X \in \Sigma_z) . 
\end{equation}

\begin{proof}[Proof of Theorem~\ref{thm:McDGenExt}]
We modify the proof of Theorem~\ref{thm:McDExt}, where independence 
is only used to establish \eqref{eq:MBCI}. Using \eqref{eq:Conv} and 
that the bijection $\rho_k:\Sigma_a \to \Sigma_b$ satisfies 
\eqref{eq:PGTLE}, we obtain 
\begin{equation*}\label{eq:MBCIGV}
\begin{split}
|\DY_{k}(a,b)| &= |\EE ( f(X) \mid X \in \Sigma_{a})-\EE ( f(X) \mid X \in \Sigma_{b})|\\
&= |\sum_{x \in \Sigma_{a}} f(x) \PP(X=x \mid X \in \Sigma_{a})-\sum_{x \in \Sigma_{b}} f(x) \PP(X=x \mid X \in \Sigma_{b})|\\
& \leq \sum_{x \in \Sigma_{a}} |f(x)-f(\rho_k(x))| \PP(X=x \mid X \in \Sigma_{a}) , 
\end{split}
\end{equation*}
which is the natural analogue of \eqref{eq:MBCI}. 
The remainder of the argument carries over with minor modifications. 
Indeed, proceeding as in \eqref{eq:MBCI2} (applying \eqref{eq:fGTL} 
instead of \eqref{eq:fTL}) and then appealing to \eqref{eq:Conv}, we 
infer 
\begin{equation}\label{eq:MBCI2GV}
\begin{split}
|\DY_{k}(a,b)| & \leq c_k + (d_k-c_k) \PP(X \not\in \Gamma \mid X \in \cF_{k-1},X_k=a) . 
\end{split}
\end{equation}
Now, by arguing as in \eqref{eq:MBC}, when $T \ge k$ holds we also have 
\begin{equation}\label{eq:MBCGV}
\begin{split}
|\DY_{k}| & \leq c_k + (d_k-c_k)\PP(X \not\in \Gamma \mid X \in \cF_{k-1}) \leq \Delta_k , 
\end{split}
\end{equation}
completing the proof. 
\end{proof}
Remark~\ref{rem:GTL} follows by similar reasoning (noting that the 
proof of Remark~\ref{rem:McDExt} carries over and that 
\eqref{eq:MBCI2GV} equals \eqref{eq:EfGTL} after replacing 
$(d_k-c_k)$ with $r_k$).

\begin{proof}[Proof of Theorem~\ref{thm:McDGenV}]
We modify the proof of Remark~\ref{rem:McDExtV} by picking (some) 
$\cF_{k-1}$-measurable $\beta \in \Lambda_k$ maximizing 
$\PP(X_k = \beta \mid \cF_{k-1})$. By assumption we have 
$\PP(X_k=\beta \mid \cF_{k-1}) \geq 1-p_k$, which in turn yields 
$\PP(D_k \neq 0 \mid \cF_{k-1}) \leq \PP(X_k \neq \beta \mid \cF_{k-1}) \le p_k$. 
Noting that all remaining applications of independence are already 
covered by \eqref{eq:MBCI2GV} and \eqref{eq:MBCGV}, this completes 
the proof. 
\end{proof}

\begin{proof}[Proof of Theorem~\ref{thm:McDGenfTL}]
With the above modifications in mind the proof of 
Theorem~\ref{thm:McDfTLP} carries over word by word, which 
establishes the first part of the claim. 
Turning to the second part, our earlier discussion shows that for 
$\Sigma_z,\Sigma_\eta \subseteq \Sigma$ with 
$|\Sigma_z|,|\Sigma_\eta|>0$ there is a bijection 
$\rho_k:\Sigma_z \to \Sigma_\eta$. So, since all possible outcomes 
occur with the same probability, we obtain 
$\PP(X \in \Sigma_z)=\PP(X \in \Sigma_\eta)$. For $\eta \in \Lambda_k$
 satisfying $|\Sigma_\eta|>0$ it follows that 
\[
\frac{1}{\PP(X_k=\eta \mid X \in \Sigma)} = \frac{\PP(X \in \Sigma)}{\PP(X \in \Sigma_\eta)} = \sum_{z \in \Lambda_k} \frac{\PP(X \in \Sigma_z)}{\PP(X \in \Sigma_\eta)} \leq |\Lambda_k| . 
\]
We deduce that 
$\min_{\eta \in \Lambda_k}\PP(X_k=\eta \mid X_1, \ldots, X_{k-1}) \geq |\Lambda_k|^{-1}$, so $q_k \leq |\Lambda_k|^{-1}$ 
suffices. 
\end{proof}

\section{Final number of edges in the reverse $H$-free process}\label{sec:Hfree}
In our analysis of the reverse $H$-free process we use several 
equivalent definitions (with respect to the final graph). 
Recall that, starting with the complete graph on vertex set $[n]$, in 
each step an edge is removed, chosen uniformly at random from all 
edges contained in a copy of $H$. 
As in~\cite{ErdoesSuenWinkler1995,Makai2012}, a moment's thought 
reveals that we may instead traverse all $\binom{n}{2}$ edges in 
random order, each time removing the current edge if and only if it 
is contained in a copy of $H$ in the evolving graph. 
As observed by Erd{\H o}s, Suen and Winkler~\cite{ErdoesSuenWinkler1995}, 
after considering $e_{\binom{n}{2}}, \ldots, e_{i+1}$ the decision 
whether $e_i$ is removed depends \emph{only} on the later edges 
$e_{i-1}, \ldots, e_{1}$ (all other `surviving' ones are by 
construction not contained in a copy of $H$). 
This allows us to consider the edges in reverse order, where $e_i$ is 
added if and only if it does not complete a copy of $H$ together with 
$e_1, \ldots, e_{i-1}$ (it does not matter whether these were added or 
not). Given a random permutation, we denote the corresponding random 
graph process after $i$ steps by $G_{n,i}(H) \subseteq G_{n,i}$, 
where $G_{n,i}$ is the uniform random graph with $n$ vertices and $i$ 
edges.

For technical reasons it will be convenient to also consider a 
continuous variant of the above process, where each edge is 
independently assigned a uniform birth time $B_e \in [0,1]$; the 
edges are then traversed in ascending order of their birth times 
(which are all distinct with probability one). The resulting process 
that considers only those edges with $B_e \leq p$ is denoted by 
$G_{n,p}(H) \subseteq G_{n,p}$. So for $p=1$ all edges are traversed 
in random order, and it follows that 
\begin{equation}
\label{eq:GnmHGnpH}
G_{n,\binom{n}{2}}(H) = G_{n,1}(H) .
\end{equation}
Conditioned on $B_e=q$, the decision whether $e$ is added only 
depends on the edges $f$ with $B_f \leq q$, which have the same 
distribution as $G_{n,q}$. As noted by Makai~\cite{Makai2012}, this 
allows for the use of classical random graph theory when estimating 
the probability that an edge is added to the evolving graph. 
Recall that $m_2(H)=d_2(H)$ for $2$-balanced graphs $H$. For 
 \[
m = n^{2-1/m_2(H)} (\log n)^{2} \quad \text{and} \quad p= n^{-1/m_2(H)} (\log n)^{2}
\]
the next lemma follows from the results of Spencer~\cite{Spencer1990} 
mentioned in Section~\ref{sec:Janson}. Note that in $G_{n,p}$ every 
pair of vertices is expected to have $\Theta((\log n)^{2(e_H-1)})$ 
`extensions' to copies of $H$. 
\begin{lemma}\label{lem:DI}
Let $H$ be a $2$-balanced graph. Let $\cD$ (and $\cI$) denote the 
event that for every pair $xy$ of vertices the following holds: after 
adding the edge $xy$ there are at most $\Psi_H=(\log n)^{2e_H}$ 
copies (is at least one copy) of $H$ containing the edge $xy$. 
For every $c>0$ we have $\PP(G_{n,m} \in \cD \cap \cI) \geq 1-n^{-c}$ 
for $n \geq n_0(c,H)$. \qed 
\end{lemma}
The point is that whenever $\cI$ holds no further edges are added. 
This allows us to couple both variants of the reverse $H$-free 
process such that they agree with very high probability after 
considering only $m$ edges. So for our purposes they are 
interchangeable, and we obtain the corresponding formal statement 
by combining Lemma~\ref{lem:DI} with \eqref{eq:GnmHGnpH}. 
\begin{lemma}\label{lem:Cpl}
Let $H$ be a $2$-balanced graph. There is a coupling such that for 
every $c>0$ we have 
\begin{equation*}
\label{eq:CplAll} 
G_{n,m}(H) = G_{n,\binom{n}{2}}(H) = G_{n,1}(H) 
\end{equation*} 
with probability at least $1-n^{-c}$ for $n \geq n_0(c,H)$. \qed
\end{lemma}

Turning to the number of edges in $G_{n,m}(H)$, which we denote by 
$e(G_{n,m}(H))$, recall that each $e_i$ is added if and only if it 
does not complete a copy of $H$ together with $e_1, \ldots, e_{i-1}$. 
So one edge can, in the worst case, influence the decisions of up to 
$O(\min\{m,n^{v_H-2}\})$ edges (whether they are added or not); 
however, on the `typical' event $\cD$ of Lemma~\ref{lem:DI} this 
is limited to at most $e_H \cdot \Psi_H = O((\log n)^{2 e_H})$ edges. 
For this reason the standard bounded differences inequality 
\emph{fails} to give useful bounds (due to large worst case $c_k$), 
whereas a \emph{routine} application of the typical bounded 
differences inequality yields sharp concentration, illustrating its 
ease of use and effectiveness. 
\begin{theorem}\label{thm:Hcon}
Let $H$ be a $2$-balanced graph. 
For every $c>0$ and $n \geq n_0(c,H)$ we have 
\begin{equation}
\label{eq:DevErr}
\PP(|e(G_{n,m}(H))-\EE e(G_{n,m}(H))| \geq \sqrt{m} (\log n)^{3 e_H}) \leq n^{-c} . 
\end{equation}
\end{theorem}
\begin{proof}
Lemma~\ref{lem:DI} implies that $G_{n,m} \in \cD$ holds with 
probability at least $1-n^{-(2c+6)}$. Note that the random sequence 
of edges $\underline{e}=(e_{1},\ldots,e_{m})$ corresponds to the 
(uniform) random graph process $(G_{n,i})_{0 \leq i \leq m}$ and 
uniquely determines $f(\underline{e})=e(G_{n,m}(H))$. 
The crucial observation is that whenever $G_{n,m},\tilde{G}_{n,m} \in \cD$ 
have edge sequences $\underline{e},\underline{\tilde{e}}$ that differ 
only in one edge (i.e., $e_j \neq \tilde{e}_j$) or the order of two 
edges (i.e., $e_j = \tilde{e}_k$ and $e_k = \tilde{e}_j$), then our 
earlier observations imply $|e(G_{n,m}(H))-e(\tilde{G}_{n,m}(H))| \leq 2e_H(\log n)^{2 e_H}=\Delta$. 
The point is that by the discussion of Section~\ref{sec:GL} this is 
exactly the condition that needs to be checked in order to apply 
Theorem~\ref{thm:McDGenExt} (with the two-sided Lipschitz condition 
\eqref{eq:fGTLP} of Theorem~\ref{thm:McDGenfTL}) using $N=m$, the 
`good' event $\Gamma=\cD$, Lipschitz coefficients $c_k = \Delta$, 
$d_k=n^2$ and the `two-sided parameter' $q_k=n^{-2}$. 
For the `compensation factor' $\gamma_k=n^{-4}$ we have 
$e_k \leq 2\gamma_k d_k q_k^{-1} \leq 2$, $c_k+e_k =\Theta((\log n)^{2e_H})$ 
and $\sum_k \gamma_k^{-1} \leq n^6$. So, using \eqref{McDExt:Pr} and 
\eqref{McDExt:PrB} we deduce that the left hand side of \eqref{eq:DevErr} 
is at most $e^{-\Omega((\log n)^{2 e_H})} + n^{-2c} \leq n^{-c}$. 
\end{proof}

To establish Theorem~\ref{thm:revH} it remains to bound the expected 
final number of edges up to constant factors. Our argument is 
inspired by Makai~\cite{Makai2012}, who proved asymptomatically 
matching bounds in~\eqref{eq:EE} for the class of strictly 
$2$-balanced graphs (the case $H=K_3$ is due to Erd{\H o}s, Suen and 
Winkler~\cite{ErdoesSuenWinkler1995}). In fact, here we determine the 
correct order of magnitude for all graphs. 
\begin{theorem}\label{thm:HE}
Let $H$ be a graph with $e_H \geq 1$. There are $a,A > 0$ such that 
\begin{equation}
\label{eq:EE}
\floor{a n^{2-1/m_2(H)}} \leq \EE e(G_{n,1}(H)) \leq A n^{2-1/m_2(H)} 
\end{equation}
for $n \geq n_0(H)$, where the floor function is only needed when $e_H=1$. 
\end{theorem}
\begin{proof}
When $e_H=1$ we have $\EE e(G_{n,1}(H))=0$ and $m_2(H)=1/2$, so 
\eqref{eq:EE} holds with, say, $a=A=1/2$. Henceforth we assume 
$e_H \ge 2$. 
Define $\cZ_e=\cZ_e(H)$ as the event that the edge $e$ is contained 
$G_{n,1}(H)$. Let $Y_{e,H,q}$ count the number of copies of $H$ in 
$G_{n,q} \cup e$ (the graph obtained by inserting $e$ into $G_{n,q}$ 
if it is not already present). Recall that, conditioned on $B_e=q$, 
only edges $f$ with $B_f \leq q$ are relevant for $\cZ_e$; so~$e$ is 
added if and only if $Y_{e,H,q}=0$. Hence for $q \in [0,1]$ we have 
\begin{equation}
\label{eq:ErrGoal:Equiv}
\PP(\cZ_e \mid B_e=q) = \PP(Y_{e,H,q}=0) . 
\end{equation}

For the lower bound in~\eqref{eq:EE} fix $F \subseteq H$ with 
$d_2(F)=m_2(H)$ that satisfies $e_F \ge 2$ (this choice is possible 
as $e_H \geq 2$). Given $e$ there are at most $Dn^{v_F-2}$ extensions 
to $F$ for some $D=D(F)>0$, so whenever $q \leq n^{-1/m_2(H)}$ holds 
monotonicity and Harris' inequality~\cite{Harris1960} yield 
\begin{equation}\label{eq:EE:LB}
\PP(Y_{e,H,q}=0) \geq \PP(Y_{e,F,q}=0) \geq (1-q^{e_F-1})^{Dn^{v_F-2}} \geq e^{-2Dn^{v_F-2}q^{e_F-1}} \geq e^{-2D} . 
\end{equation} 
Together with \eqref{eq:ErrGoal:Equiv} we obtain 
$\PP(\cZ_e) = \EE \: \PP(\cZ_e \mid B_e=q) \geq n^{-1/m_2(H)} \cdot e^{-2D}$, 
and the lower bound in~\eqref{eq:EE} now follows by linearity of 
expectation.

Turning to the upper bound in~\eqref{eq:EE}, consider 
$q=\lambda n^{-1/m_2(H)}$ with $1 \leq \lambda \leq n^{1/m_2(H)}$. 
We apply Janson's inequality to $Y_{e,H,q}$, which counts the number 
of extensions of $e$ to $H$ (viewed as subgraphs these do not contain 
the edge $e$). Note that $e_H \geq 2$, $\lambda \geq 1$ and 
$m_2(H) \geq (e_H-1)/(v_H-2)$ imply 
\[
\mu=\EE Y_{e,H,q} = \Theta(n^{v_H-2}q^{e_H-1})= n^{(v_H-2)-(e_H-1)/m_2(H)}\Theta(\lambda^{e_H-1}) = \Omega(\lambda) . 
\]
Define $\cG$ as the set of all proper subgraphs graphs 
$G \subsetneq H$ with $e_G \geq 2$. Considering all possible 
`overlaps' of extensions of $e$ to $H$ (analogous to the textbook 
proof of the small subgraphs theorem), the $\Delta$ term of Janson's 
inequality satisfies 
\begin{equation*}
\begin{split}
\Delta & \leq O(n^{v_H-2}q^{e_H-1}) \cdot \sum_{G \in \cG} O(n^{v_H-v_G}q^{e_H-e_G}) = O(\mu^2) \sum_{G \in \cG} n^{-(v_G-2)}q^{-(e_G-1)} \\ 
& = O(\mu^2) \sum_{G \in \cG} n^{-[(v_G-2)-(e_G-1)/m_2(H)]}\lambda^{-(e_G-1)} = O(\mu^2/\lambda) , 
\end{split}
\end{equation*}
where the last inequality follows from $e_G \geq 2$, $\lambda \geq 1$ 
and $m_2(H) \geq (e_G-1)/(v_G-2)$. So, using $\mu/\lambda = \Omega(1)$ 
we infer $\mu+\Delta=O(\mu^2/\lambda)$ and thus 
$\mu^2/(\mu+2\Delta) = \Omega(\lambda)=\Omega(n^{1/m_2(H)}q)$. 
Applying Janson's inequality (see e.g.\ Theorem~2.18 in~\cite{JLR}) 
we have $\PP(Y_{e,H,q}=0) \leq \exp\left(-C n^{1/m_2(H)}q\right)$ for 
$C=C(H)>0$. Combining this with \eqref{eq:ErrGoal:Equiv} when 
$q \geq n^{-1/m_2(H)}$ and the trivial bound 
$\PP(\cZ_e \mid B_e=q) \leq 1$ otherwise, for $A=1+e^{-C}/C$ we 
obtain 
\begin{equation}\label{eq:EE:UB}
\PP(\cZ_e) = \EE \: \PP(\cZ_e \mid B_e=q) \leq n^{-1/m_2(H)} + \int_{n^{-1/m_2(H)}}^{1}\exp\left(-C n^{1/m_2(H)}q\right) dq \leq A n^{-1/m_2(H)} . 
\end{equation}
Linearity of expectation now yields the upper bound in~\eqref{eq:EE}. 
\end{proof}

Our arguments partially generalize to arbitrary graphs, which we 
shall now briefly discuss. In this case Lemma~\ref{lem:DI} remains 
true if we modify $\cD$ to at most, say, $\Psi_H=(\log n)n^{v_H-2}p^{e_H-1}$ 
copies, and so the coupling of Lemma~\ref{lem:Cpl} carries over (it 
only uses $\cI$). With \eqref{eq:GnmHGnpH} in mind, Theorem~\ref{thm:HE} 
shows that the expected final number of edges is $\mu=\Theta(n^{2-1/m_2(H)})$. 
Adjusting the proof of Theorem~\ref{thm:Hcon} with $c_k=2e_H\Psi_H$, 
a short calculation shows that we obtain concentration on an interval 
of length $\mu n^{-\gamma}$ with $\gamma=\gamma(H)>0$ whenever 
\begin{equation}\label{eq:ExtBal}
\text{$v_H \geq 4$ and $m_2(H) < (2e_H-3)/(2v_H-6)$} \quad \text{ or } \quad \text{$v_H=3$ and $e_H \geq 2$} . 
\end{equation}
Perhaps surprisingly, this condition is satisfied by standard 
examples of `unbalanced' graphs such as a clique $K_r$ with an extra 
edge hanging off.

The proofs in this section also extend with minor modifications to 
the more general reverse $\cH$-free process considered in 
Theorem~\ref{thm:revFH}. In this case the `inverted' processes 
$\cG_{n,m}(\cH)$ and $\cG_{n,p}(\cH)$ are defined in analogous ways, 
where an edge is added only when it closes no copy of some 
$F \in \cH$. We need to modify $\cD$ of Lemma~\ref{lem:DI} so that 
for all $F \in \cH$ it ensures at most $\Psi_F=\max\{(\log n) n^{v_F-2}p^{e_F-1},(\log n)^2\}$ 
copies, whereas the corresponding $\cI$ only applies to the 
distinguished graph $H$ with $m_2(H)=d_2(H)=\min_{F \in \cH}d_2(F)$. 
As before, once $\cI$ holds no more edges are added. With this in 
mind the coupling of Lemma~\ref{lem:Cpl} as well as the 
concentration result of Theorem~\ref{thm:Hcon} carry over in a 
straightforward way (noting that $d_2(H) \leq d_2(F)$ implies 
$\Psi_F \leq (\log n)^{2e_F}$ for all $F \in \cH$). Turning to the 
expected final number of edges, for the lower bound of 
Theorem~\ref{thm:HE} we avoid all $F \in \cH$ simultaneously. 
The resulting modification of \eqref{eq:EE:LB} works for 
$q \leq n^{-1/m_2(H)}$ since $d_2(H) \leq d_2(F)$ implies 
$n^{v_{F}-2}q^{e_{F}-1} \leq 1$. For the upper bound it suffices to 
just avoid the distinguished $2$-balanced graph $H$, so we may reuse 
the estimates of~\eqref{eq:EE:UB} to establish Theorem~\ref{thm:revFH}.

Finally, note that every edge added by $G_{n,m}(H)$ is also added by 
the $H$-free process defined in Section~\ref{sec:HfreeIntro} (where 
$e_i$ is added if and only if it does not complete a copy of $H$ 
together with the \emph{added} edges among $e_1, \ldots, e_{i-1}$). 
It follows from Theorem~\ref{thm:HE} that the expected final number 
of edges in the $H$-free process is at least $\Omega(n^{2-1/m_2(H)})$ 
for any graph $H$, which improves the $\Omega(n^{2-1/d_2(H)})$ bound 
resulting from the deletion argument of Osthus and Taraz~\cite{OsthusTaraz2001}. 
In fact, if the technical conditions in~\eqref{eq:ExtBal} are 
satisfied our earlier discussion implies that this lower bound also 
holds with probability tending to one (not only in expectation), 
which for `unbalanced' graphs with $m_2(H) > d_2(H)$ does not follow 
from Theorem~1 in~\cite{OsthusTaraz2001}.

\bigskip{\bf Acknowledgements.} 
I am grateful to my supervisor Oliver Riordan for a very careful 
reading of an earlier version of this paper, and for many helpful 
comments. I would also like to thank Tam{\'a}s Makai for sending me 
a preprint of \cite{Makai2012}, Matas {{\v{S}}ileikis for remarks, 
and Colin McDiarmid for asking whether Theorem~\ref{thm:McDExt} 
extends to random permutations.

\small\begin{spacing}{0.4}
\bibliographystyle{plain}

\end{spacing}
\normalsize
\end{document}